\documentclass[11pt]{article}
\usepackage{amssymb,amsmath,amsthm,secdot,bbm,mathrsfs,mathtools}

\usepackage[
bookmarks=true,
bookmarksnumbered=true,
colorlinks=true, pdfstartview=FitV, linkcolor=blue, citecolor=blue,
urlcolor=blue]{hyperref}

\usepackage{graphicx}

\usepackage[shortlabels]{enumitem}

\usepackage{color}

\usepackage{microtype}
\usepackage[nameinlink,capitalize]{cleveref}

\usepackage[a4paper,innermargin=1.4in,outermargin=1.4in,
bottom=1.5in,marginparwidth=1in,marginparsep
=3mm]{geometry}

\numberwithin{equation}{section}

\usepackage{hyperref}
\hypersetup{
	colorlinks=true,
	linkcolor=blue,
	citecolor=red,
	urlcolor=blue,
	pdfborder={0 0 0}
}

\usepackage[nameinlink,capitalize]{cleveref}

\crefname{section}{Section}{Sections}
\crefname{figure}{Figure}{Figures}
\crefname{theorem}{Theorem}{Theorems}
\crefname{lemma}{Lemma}{Lemmas}
\crefname{proposition}{Proposition}{Propositions}
\crefname{corollary}{Corollary}{Corollaries}
\crefname{definition}{Definition}{Definitions}
\crefname{example}{Example}{Examples}
\crefname{remark}{Remark}{Remarks}

\newtheorem{theorem}{Theorem}[section]
\newtheorem{lemma}[theorem]{Lemma}

\newtheorem{corollary}[theorem]{Corollary}

\theoremstyle{definition}

\theoremstyle{remark}
\newtheorem{remark}[theorem]{Remark}

\providecommand*{\abs}[1]{\lvert#1\rvert}
\providecommand*{\norm}[1]{\lVert#1\rVert}

\newcommand*{\defeq}{\mathrel{\mathop:}=}

\renewcommand*{\Re}{\operatorname{Re}}
\renewcommand*{\Im}{\operatorname{Im}}

\newcommand*{\hcap}{\operatorname{hcap}}

\renewcommand{\ge}{\geqslant}
\renewcommand{\le}{\leqslant}
\newcommand{\nn}{\nonumber}

\newcommand{\wt}{\widetilde}
\newcommand{\wh}{\widehat}
\usepackage{accents}
\newcommand{\dbtilde}[1]{\accentset{\approx}{#1}}

\DeclareMathOperator{\Law}{Law}
\DeclareMathOperator{\spn}{span}
\DeclareMathOperator\supp{supp}
\DeclareMathOperator*{\argmin}{arg\,min} 
\DeclareMathOperator*{\argmax}{arg\,max} 

\def\E{\hskip.15ex\mathsf{E}\hskip.10ex}
\def\P{\mathsf{P}}

\def\eps{\varepsilon}
\def\phi{\varphi}

\renewcommand{\d}{\partial}

\newcommand*{\pr}{\P}
\newcommand*{\ex}{\E}

\newcommand{\CC}{\mathcal{C}}
\renewcommand{\H}{\mathbb H}
\newcommand{\R}{\mathbb{R}}
\newcommand{\Z}{\mathbb{Z}}
\newcommand{\N}{\mathbb{N}}

\newcommand*{\NN}{\mathbb N}

\newcommand*{\RR}{\mathbb R}
\newcommand*{\HH}{\H}
\newcommand*{\barH}{\overline\HH}

\newcommand*{\ind}{\mathbbm 1}
\newcommand*{\slek}{SLE$_\kappa$}

\title{Law of the SLE tip}
\author{Oleg Butkovsky%
	\thanks{Weierstrass Institute, Mohrenstrasse 39, 10117 Berlin, FRG. Email: \texttt{oleg.butkovskiy@gmail.com}}
		\setcounter{footnote}{3}
	\and
	Vlad Margarint%
	\thanks{University of North Carolina at Charlotte. Email: \texttt{vmargari@uncc.edu}}
	\and
	Yizheng Yuan
	\thanks{Technische Universit\"at Berlin, Stra\ss{}e des 17. Juni 136, 10623 Berlin. Email: \texttt{yuan@math.tu-berlin.de}}
}

\begin{document}

\maketitle

\begin{abstract}
We analyse the law of the SLE tip at a fixed time in capacity parametrization. We describe it as the stationary law of a suitable diffusion process, and show that it has a density which is the unique solution (up to a multiplicative constant) of a certain PDE.
Moreover, we identify the phases in which the even negative moments of the imaginary value are finite. For the negative second and negative fourth moments we provide closed-form expressions.

\end{abstract}

\section{Introduction}
The Schramm-Loewner evolution (SLE$_\kappa$) is a family of random planar fractal curves indexed by the real parameter $\kappa\geq 0$, introduced by Schramm in \cite{schramm2000scaling}. These random fractal curves are proved to describe scaling limits of a number of discrete models that are of great interest in planar statistical physics. For instance, it was proved in \cite{lsw2004conformal}  that the scaling limit of loop-erased random walk (with the loops erased in a chronological order) converges in the scaling limit to SLE$_\kappa$ with  $\kappa = 2\,.$ Moreover, other two-dimensional discrete models from statistical mechanics including Ising model cluster boundaries, Gaussian free field interfaces, percolation on the triangular lattice at critical probability, and uniform spanning tree Peano curves were proved to converge in the scaling limit to SLE$_\kappa$ for values of $\kappa=3,$ $\kappa=4,$ $\kappa=6$ and $\kappa=8$ respectively in the series of works \cite{smirnov2010conformal}, \cite{schramm2009contour}, \cite{smirnov2001critical}  and \cite{lsw2004conformal}. There are also other models of statistical physics in 2D that are conjectured to have SLE$_\kappa$, for some value of $\kappa$, as a scaling limit, among which is the two-dimensional self-avoiding walk which is conjectured to converge in the scaling limit to SLE$_{8/3}$.
For a detailed exposure and pedagogical introduction to SLE theory, we refer the reader to \cite{rs2005basic}, \cite{lawler2005conformally}, and \cite{kemppainen2017schramm}.

Questions concerning the behaviour of the SLE trace at the tip can be found in the existing body of SLE literature, for example in \cite{viklund2012almost} where the almost sure multi-fractal spectrum of the SLE trace near its tip is computed, and in \cite{zhan2016ergodicity} in which the ergodic properties of the harmonic measure near the tip of the SLE trace are studied.


However, to the best of our knowledge, the law of the SLE tip at fixed capacity time has not been studied in the SLE literature until very recently. One of the first papers in this direction is \cite{lyons2019convergence} where a method based on  stopping times was applied in order to try to 
deduce information about the   law of the SLE tip.

In this article, we develop an approach that allows for an in-depth study of this fundamental quantity. More precisely, we derive a PDE whose unique solution is the density of the SLE tip. This allows us to obtain explicit values for the negative second and negative fourth moment of the imaginary value of the SLE tip. We deduce that they are finite only for $\kappa< 8$ resp. $\kappa < 8/3$. For further negative moments, we identify the values of $\kappa$ where the moments are finite. 

To obtain these results we combine PDE techniques with certain tools from the theory of stochastic stability of stochastic differential equations (SDEs). Namely, we work with an SDE obtained from the backward Loewner differential equation. By a scaling argument, we derive a two-dimensional diffusion process that converges in law to the SLE tip. Using tools from ergodic theory (in the spirit of \cite{MSH}), we prove that this diffusion process has a unique invariant measure.  This allows us to show that the density of the SLE tip solves the Fokker-Planck-Kolmogorov (FPK) equation associated with the process.

Showing that the density of SLE tip is the unique solution of the FPK equation requires further tools. Note that while there is a vast literature on FPK equations (see e.g. \cite{BKRS}), usually only the case of elliptic operators are considered, while our FPK is hypoelliptic. Therefore, to show uniqueness of solutions to this equation and derive the support of the solution we utilise the generalized  Ambrosio-Figalli-Trevisan superposition principle obtained recently in \cite{BRS} as well as more standard methods such as Lyapunov functions and Harnack inequalities.

This paper is organised in three sections, the first one being the introduction. In the second section we state the main results. In the last section which is further divided in two subsections we give their proofs.

\medskip
\noindent\textbf{Convention on constants.}  Throughout the paper $C$ denotes a positive constant whose
value may change from line to line.

\medskip
\noindent\textbf{Acknowledgements.} The authors are  deeply indebted to Stas Shaposhnikov for his help, patience, detailed explanations of some parts of the theory of FPK equations and for suggesting some useful ideas for the proofs. We are very grateful to Paolo Pigato and Peter Friz for fruitful discussions. We also would like to express our deep gratitude to the referee for thoroughly reading the paper and for offering very valuable suggestions. OB has received
funding from the European Research Council (ERC) under the European Union’s Horizon 2020 research and innovation program (grant agreement No. 683164), from the DFG Research Unit FOR 2402, and is funded by the Deutsche Forschungsgemeinschaft (DFG, German Research Foundation) under Germany's Excellence Strategy --- The Berlin Mathematics Research Center MATH+ (EXC-2046/1, project ID: 390685689, sub-project EF1-22). YY acknowledges partial support from ERC through Consolidator Grant 683164 (PI: Peter Friz).

\section{Main results}

First, let us introduce the basic notation. For a domain $D\subset\R^k$, $k\ge1$, let $\CC^\infty(D,\R)$ be a set of functions $D\to\R$ which have derivatives of all orders. The set of functions from $\CC^\infty(D,\R)$  which are bounded and have bounded derivatives of all orders will be denoted by $\CC_b^\infty(D,\R)$.  As usual, for a function  $f\colon D\to\R^d$, $d\ge1$, we will denote its supremum norm by 
$\|f\|_{\infty}:=\sup_{x\in D} |f(x)|$. 
Let $\H$ be  the open complex upper half-plane $\{\Im(z) > 0\}$.

Until the end of the paper we fix $\kappa\in(0,\infty)$. Let $g_t\colon H_t\to\H$, $t\ge0$, be the forward SLE flow, that is the solution to the Loewner ODE
\begin{equation*}
\partial_t g_t(z) = \frac{2}{g_t(z)-\sqrt{\kappa} B_t}, \quad g_0(z)=z,\quad t\ge0,\, z\in\H,
\end{equation*}
where $B$ is a standard Brownian motion, $H_t = \{z \in \HH\mid T_z > t\}$, and $T_z$ is the time until which the ODE is solvable. Let $(\gamma_t)_{t\ge0}$ be the \slek\ path associated with this flow. It is well-known  \cite[Theorem 3.6]{rs2005basic}, \cite[Theorem 4.7]{lsw2004conformal} that $\P$-a.s. for any $t\ge0$
\begin{equation}\label{SLElim}
\gamma(t)=\lim_{u\to0+}g_t^{-1}(\sqrt\kappa B_t+iu).
\end{equation}
Throughout the paper we use the notations $\gamma(t)$ and $\gamma_t$ interchangeably.

Our main result is the following statement.

\begin{theorem}\label{T:main}
The random vector $(\Re(\gamma_1),\Im(\gamma_1))$ has a density $\psi\in\CC^{\infty}(\R\times(0,\infty),\R)$ which is the unique solution in the class of probability densities (non-negative functions that integrate to $1$ over the whole space) of the following PDE:
\begin{equation}\label{mainPDE}
\frac{\kappa}{2}\d^2_{xx}\psi+\left(\frac12x+\frac{2x}{x^2+y^2}\right)\d_x \psi+
\left(\frac12y-\frac{2y}{x^2+y^2}\right)\d_y\psi+\left(1+\frac{4(y^2-x^2)}{(x^2+y^2)^2}\right)\psi = 0,
\end{equation}
where $x\in\R$, $y\in(0,\infty)$.

Furthermore, $\psi$ is strictly positive in $\R\times(0,2)$, $\psi\equiv 0$ on $\R\times[2,\infty)$, and $\psi(x,y)=\psi(-x,y)$ for $x\in\R$, $y>0$. 
\end{theorem}

We have attached in \cref{fi:simul} numerical simulations of $\gamma(1)$ with various values of $\kappa$. There we have chosen the coordinates $(\alpha,y)$ where $\alpha = \arg\gamma(1)$ and $y=\Im\gamma(1)$ so that they fit well in the plot.

\begin{figure}[h]
	\centering
	\includegraphics[width=0.49\textwidth]{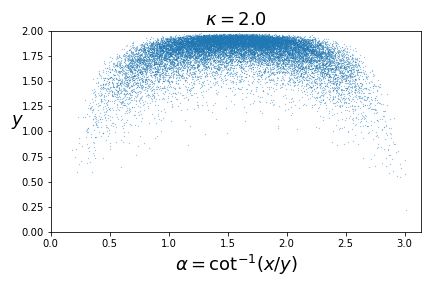}
	\includegraphics[width=0.49\textwidth]{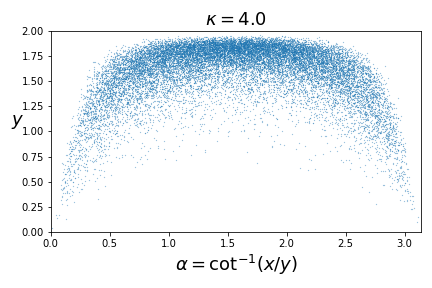}
	\includegraphics[width=0.49\textwidth]{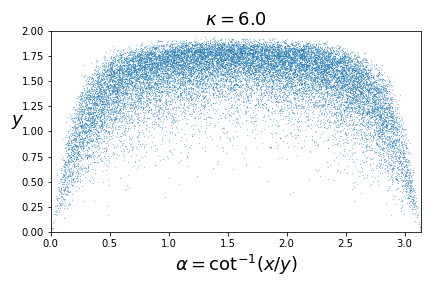}
	\includegraphics[width=0.49\textwidth]{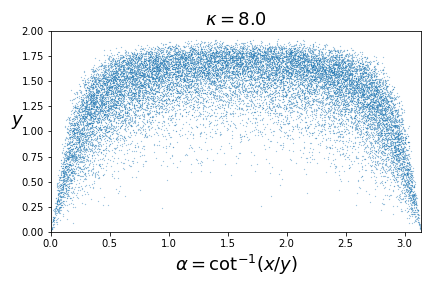}
	\caption{Simulation of $\gamma(1)$ with $20000$ samples each. Plotted are the coordinates $(\alpha,y)$ where $\alpha = \arg\gamma(1)$ and $y=\Im\gamma(1)$.}
	\label{fi:simul}
\end{figure}

As an application of \cref{T:main}, we show that the following quantities can be explicitly calculated.

\begin{theorem}\label{T:other}
The following holds:
\begin{enumerate}[\rm{(}i\rm{)}]
\item For any measurable set $\Lambda\subset\barH$ one has
\begin{equation}\label{eq:zhan_law}
		\ex \int_0^\infty \ind(\gamma(t) \in \Lambda) \, dt = \frac{\Gamma(1+\frac4\kappa)}{2\sqrt\pi\Gamma(\frac12+\frac4\kappa)} \int_\Lambda \left( 1+\frac{x^2}{y^2} \right)^{-4/\kappa} \,dx\,dy.
\end{equation}
\item For any $n\in\N$ we have 
\begin{equation}\label{ineq24}
\E (\Im\gamma_1)^{-2n}<\infty \,\,\text{if and only if $\kappa<8/(2n-1)$}.
\end{equation}
Further,
\begin{align}
&\E(\Im\gamma_1)^{-2}=\frac{2}{8-\kappa}\quad\text{for $\kappa < 8$}\label{ineq25} \\
&\E(\Im\gamma_1)^{-4}=\frac{16(3-\kappa)}{(12-\kappa)(8-\kappa)(8-3\kappa)}\quad\text{for $\kappa < 8/3$}.\label{ineq26}
\end{align}
\end{enumerate}
\end{theorem}

\begin{remark}
Note that the left-hand side of \eqref{eq:zhan_law} is an average amount of time SLE spends in a set $\Lambda$. A version of this identity has previously appeared in \cite[Corollary 5.3]{zhan2019decomposition}. However, in that paper the constant in front of the integral has been implicitly specified as $1/C_{\kappa,1}$ with
    \[ C_{\kappa,1} = \int_\HH \left( M_0(z)-\ex[M_1(z)\ind_{T_z > 1}] \right) \,dx\,dy \]
    and $M_t(z) = \abs{g_t'(z)}^2 \left(\frac{\Im g_t(z)}{\abs{g_t(z)-\sqrt{\kappa}B_t}}\right)^{8/\kappa}$. In particular, our result implies
    \[ C_{\kappa,1} = \frac{2\sqrt\pi\Gamma(\frac12+\frac4\kappa)}{\Gamma(1+\frac4\kappa)} \]
\end{remark}

As we will point out in \cref{se:density_analysis}, our \cref{T:other} may seem like a simple consequence of \cref{T:main} that can be heuristically deduced from integration by parts arguments. However, it is surprisingly tricky to control the boundary behaviour of $\psi$ and its derivatives. Therefore it requires more work to rigorously prove \cref{T:other}.

One of our initial motivations was to know more about the marginal law of $\alpha = \arg\gamma(1)$. We believe that the marginal density should behave like $\alpha^{8/\kappa}$ as $\alpha \searrow 0$. We did not succeed in proving this; instead, we prove the following in \cref{se:density_analysis}. Denote $(\alpha, y) = (\arg\gamma(1), \Im\gamma(1))$ and let $q(\alpha, y) = \psi(y\cot\alpha,y)\frac{y}{\sin^2\alpha}$ the density in these coordinates. Then for $n \ge 1$ we have $\int_0^2 y^{-2n} q(\alpha,y)\,dy \approx \alpha^{8/\kappa-2n}$ as $\alpha \searrow 0$.

\begin{remark}
    The support of the density is quite natural since the half-plane capacity of $\gamma[0,t]$ is always at least $\frac{1}{2}\Im\gamma(t)^2$, and hence we always have $\Im\gamma(t) \le \sqrt{2\hcap(\gamma[0,t])} = 2\sqrt{t}$. Note also that $\Im\gamma(t) = 2\sqrt{t}$ is only attained by SLE$_0$, i.e. $\gamma(t) = i2\sqrt{t}$ which is driven by the constant driving function.
\end{remark}

To obtain these results we establish the following lemma which links the law of SLE$_\kappa$ with invariant measure of a certain diffusion process. Introduce the reverse SLE flow
\begin{equation}\label{hdef}
\partial_t h_t(z) = \frac{-2}{h_t(z)-\sqrt{\kappa}\wt B_t}, \quad h_0(z)=z,
\quad t\ge0,\, z\in\H;
\end{equation}
where $\wt B$ is the time-reversed Brownian motion, that is,
\begin{equation}\label{RBM}
\wt B_t:= B_{1-t}- B_1\quad\text{for $t\le1$;}\quad  \wt B_t:= B'_{t-1}-B_1\quad\text{for $t\ge1$},	
\end{equation}
where $B'$ is a Brownian motion independent of $B$. It is obvious that $\wt B$ is a Brownian motion.

\begin{lemma}\label{L:main}We have
\begin{equation*}
\frac1{\sqrt t}(h_{t}(i)- \sqrt{\kappa}\wt B_{t})\to\gamma(1)\,\,\text{in law as $t\to\infty$}.
\end{equation*}
\end{lemma}

\section{Proofs}

\subsection{Proofs of \cref{L:main,T:main}}

We begin with the proof of \cref{L:main}.
\begin{proof}[Proof of \cref{L:main}]
Introduce $\wh f_t(z):=g_t^{-1}(\sqrt\kappa B_t+z)$, $z\in\H$, $t\ge0$. We claim that \begin{equation}\label{claim}
\wh f_1(z) = h_1(z)+\sqrt{\kappa}B_1.
\end{equation}

Indeed, it follows from \eqref{hdef} that for $z\in\H$, $t\in[0,1]$
\begin{equation*}
\partial_t(h_{1-t}(z)+\sqrt{\kappa} B_1) = \frac{2}{h_{1-t}(z)-\sqrt{\kappa}\wt B_{1-t}}
 = \frac{2}{(h_{1-t}(z)+\sqrt{\kappa} B_1)-\sqrt{\kappa} B_t},
\end{equation*}
which implies $h_{1-t}(z)+\sqrt{\kappa}B_1 = g_t(h_1(z)+\sqrt{\kappa} B_1)$. Recalling the definition of $\wh f$ and taking $t=1$, we obtain \eqref{claim}.

Next, we note that the following scaling property holds: for any $c>0$
\begin{equation}
\label{samelaw}
\Law(\frac{1}{c}\sqrt{\kappa}\wt B_{c^2 }, \frac{1}{c}h_{c^2 }(cz))=\Law(\sqrt{\kappa}\wt B_1, h_1(z))).
\end{equation} 
Indeed, using again the definition of $h$ in \eqref{hdef}, we see that for any $t\ge0$
\begin{equation*}
\partial_t (\frac{1}{c}h_{c^2 t}(cz)) = \frac{-2c}{h_{c^2 t}(cz)-\sqrt{\kappa}\wt B_{c^2 t}}
= \frac{-2}{\frac{1}{c}h_{c^2 t}(cz)-\frac{1}{c}\sqrt{\kappa}\wt B_{c^2 t}} .
\end{equation*}
Since the process $(\frac{1}{c}\sqrt{\kappa}\wt B_{c^2 t})_{t \ge 0}$ has the same law as $(\sqrt{\kappa}\wt B_t)_{t \geq 0}$ and the solution of the Loewner differential equation is a deterministic function of the driver, we see that \eqref{samelaw} holds.

Fix $u>0$. Applying \eqref{claim} with $z=iu$ and \eqref{samelaw} with $z=iu$, $c=1/u$, we deduce
\begin{equation*}
\Law(\wh f_1(iu))=\Law(u h_{\frac1{u^{2}}}(i)-u\sqrt\kappa \wt B_{\frac1{u^{2}}}).
\end{equation*} 
where we have also used the fact that $B_1=-\wt B_1$. 
Since, by \eqref{SLElim}, we have $\gamma(1) = \lim_{u \searrow 0} \wh f_1(iu)$, it follows that $u(h_{1/u^2}(i)- \sqrt{\kappa}\wt B_{1/u^2})$ converges in law to $\gamma(1)$ as $u \searrow 0$. This implies the statement of the lemma.
\end{proof}

Recall the definition of the reverse SLE flow $h$ in \eqref{hdef} and the reversed Brownian Motion $\wt B$ in \eqref{RBM}. \Cref{L:main} implies the following result.
\begin{corollary}\label{Cor:main} Let $(\wh X_t,\wh Z_t)_{t\ge0}$ be the stochastic process that satisfies the following equation 
\begin{align}
	d\wh X_t &= \bigl(-\frac12 \wh X_t -\frac{2\wh X_t}{\wh X_t^2+e^{2\wh Z_t}}\bigr)\, dt + \sqrt{\kappa}\, d\wh B_t, \label{eq11}\\
	d\wh Z_t &= \bigl(-\frac12  +\frac{2}{\wh X_t^2+e^{2\wh Z_t}}\bigr)\, dt,\label{eq22}
	\end{align}	
with the initial data $\wh X_0=\Re (h_1(i))-\sqrt\kappa \wt B_1$,
$\wh Z_0=\log(\Im (h_1(i)))$; here  $\wh B_t := -\int_0^t e^{-s/2}\, d\wt B_{e^s}$ and the filtration $\wh{\mathcal{F}}_t :=\sigma(\wt B_r, r \in [0,e^t])$. Then
\begin{equation}\label{corres}
(\wh X_t,\wh Z_t)\to (\Re(\gamma_1),\log(\Im(\gamma_1)))\quad \text{in law as $t\to\infty$}. 
\end{equation}
\end{corollary}

Note that the initial value of the process $(\wh X,\wh Z)$ is random but measurable with respect to $\wh{\mathcal{F}}_0$.
\begin{proof}
Put $X_t+iY_t \defeq h_t(i)-\sqrt{\kappa}\wt B_t$. Then, it follows from \eqref{hdef} that
\begin{align}
	dX_t &= \frac{-2X_t}{X_t^2+Y_t^2}\, dt - \sqrt{\kappa}\, d\wt B_t, \nn\\
	dY_t &= \frac{2Y_t}{X_t^2+Y_t^2}\, dt,\label{yfirst}
\end{align}
$X_0=0$, $Y_0=1$.
For $t \ge 0$, let 
$\wh X_t:=  e^{-t/2} X_{e^t}$ and $\wh Y_t := e^{-t/2} Y_{e^t}$.
We apply It\^o's formula to derive
\begin{align}
	d\wh X_t &= \bigl(-\frac12 \wh X_t -\frac{2\wh X_t}{\wh X_t^2+\wh Y_t^2}\bigr)\, dt + \sqrt{\kappa}\, d\wh B_t, \label{eq1}\\
	d\wh Y_t &= \bigl(-\frac12 \wh Y_t +\frac{2\wh Y_t}{\wh X_t^2+\wh Y_t^2}\bigr)\, dt.\label{eq2}
\end{align}
Clearly, $\wh B$ is a standard Brownian motion with respect to the  filtration $\wh{\mathcal{F}}_t$. By definition, we also have 
$\wh X_0=X_1$, $\wh Y_0=Y _1$. The change of variables $\wh Z_t:=\log\wh Y_t$ and another application of It\^o's formula implies that the process  $(\wh X,\wh Z)_{t\ge0}$ satisfies SDE \eqref{eq11}--\eqref{eq22} with the initial conditions 
$\wh X_0=X_1=\Re (h_1(i))-\sqrt\kappa \wt B_1$,
$\wh Z_0=\log (Y _1)=\log(\Im (h_1(i)))$. Note that by \eqref{yfirst},
$Y_1\ge Y_0=1$, therefore $|\wh Z_0|<\infty$. 

Furthermore, 
\begin{equation*}
	e^{-t/2} (h_{e^t}(i)- \sqrt{\kappa}\wt B_{e^t}) = e^{-t/2} (X_{e^t}+iY_{e^t}) = 
	\wh X_t +i \wh Y_t.
\end{equation*}
Thus, by \Cref{L:main}, we have
\begin{equation}\label{convhat}
(\wh X_t,\wh Y_t)\to(\Re(\gamma(1)), \Im(\gamma(1)))\,\,\text{in law as $t\to\infty$}.
\end{equation}
Note that 
\begin{equation}\label{zeroprop}
\P(\Im(\gamma(1))=0)=0.    
\end{equation}
Indeed, the trace of a Loewner chain a.s. spends zero capacity time at the boundary, i.e., $\lambda(\{t \mid \Im\gamma(t)=0 \}) = 0$ a.s., where $\lambda$ is the Lebesgue measure (cf. \cite[Proposition~1.7]{yuan2022topology}; the case for \slek{} appeared already in \cite[Corollary~5.3]{zhan2019decomposition}). Therefore, by Fubini's theorem, $\P(\Im(\gamma(t))=0)=0$ Lebesgue a.e.. By scale invariance, this implies \eqref{zeroprop}.

Now, combining \eqref{convhat} and \eqref{zeroprop}, we get \eqref{corres}.
\end{proof}

It follows from \Cref{Cor:main} that to prove \Cref{T:main} one needs to study invariant measures of \eqref{eq11}--\eqref{eq22}.  PDE \eqref{mainPDE} is then the Fokker-Planck-Kolmogorov equation for this process. However, since the coefficients have a singularity at $0$, a bit of care is needed to make the statements rigorous.

First, we show that this SDE is well-posed and is a Markov process. We will need the following notation. 
For a vector field  $U\colon\R^2\to\R^2$ denote its derivative matrix by $(DU)_{i,j}:=\partial_{x_j} U_i$. The Lie bracket between two vector fields $U,V\colon\R^2\to\R^2$ is given by 
$$
[U,V](x):=DV (x)U(x)-DU(x) V (x),\quad x\in\R^2.
$$
It is immediate to see that if $U= \left(\begin{smallmatrix}1\\0\end{smallmatrix}\right)$, then 
\begin{equation}\label{lie}
		[U, V]=\begin{pmatrix}
			\d_{x_1}V_{1}\\ \d_{x_1}V_{2} 
		\end{pmatrix},\quad
		[U,[U,V]\,]=\begin{pmatrix}
			\d^2_{x_1x_1}V_{1}\\\d^2_{x_1x_1}V_{2}
		\end{pmatrix}.\quad
	\end{equation}

We begin with the following technical statement.

Let $W$ be a standard Brownian motion.
For $\eps>0$, let $g_\eps\colon\R\to[\eps/2,+\infty)$ be a $\CC^\infty(\R)$ function with bounded derivatives of all orders such that
\begin{equation*}
	\begin{cases}
		g_\eps(x)=x,\quad &x\ge\eps;\\
		\eps/2\le g_\eps(x)\le \eps,\quad &-\infty<x<\eps.
	\end{cases}
\end{equation*}  

\begin{lemma}\label{L:31aux}
Fix $\eps>0$ and consider stochastic differential equation
\begin{align}
d X^\eps_t &= \bigl(-\frac12 X^\eps_t -\frac{2 X^\eps_t}{ (X^\eps_t)^2+ g_\eps(e^{2 Z_t^\eps})}\bigr)\, dt + \sqrt{\kappa}\, d W_t, \label{eq1eps}\\
d Z^\eps_t &= \bigl(-\frac12   +\frac{2}{ (X^\eps_t)^2+ g_\eps(e^{2 Z_t^\eps})}\bigr)\, dt,\label{eq2eps}
\end{align}
where $(X^\eps_0,Z^\eps_0)=(x_0,z_0)\in\R^2$.  Then 
for any initial condition $(x_0,z_0)\in\R^2$ SDE \eqref{eq1eps}--\eqref{eq2eps} has a unique strong solution. This solution is a strong Feller Markov process. 
\end{lemma}
\begin{proof}
Since the drift and diffusion of \eqref{eq1eps}--\eqref{eq2eps} are uniformly Lipschitz continuous functions, it is immediate that SDE \eqref{eq1eps}--\eqref{eq2eps} has a unique strong solution and this solution is a Markov process. To show that $(X_t^\eps,Z_t^\eps)$ is a strong Feller process we use (parabolic) H\"ormander's theorem.
 
Denote
\begin{equation}\label{bepss}
b^\eps(x,z):=\begin{pmatrix}
b^{1,\eps}(x,z)\\b^{2,\eps}(x,z) 
\end{pmatrix}:=\begin{pmatrix}
			-\frac12 x -\frac{2x}{x^2+g_\eps(e^{2z})}\\[1.5ex]
			-\frac12  +\frac{2}{x^2+g_\eps(e^{2z})} 
		\end{pmatrix},\quad x,z\in\R;\qquad \sigma:=\begin{pmatrix}
			\sqrt \kappa\\0 
		\end{pmatrix}.
\end{equation}
Then we can rewrite \eqref{eq1eps}--\eqref{eq2eps} as
\begin{equation}\label{Zeq}
d\xi_t^\eps=b^\eps(\xi^\eps_t)dt+\sigma d W_t,
\end{equation}
	where we put $\xi^\eps:=\begin{psmallmatrix}
		X^\eps\\Z^\eps. 
	\end{psmallmatrix}$
Let us verify that SDE \eqref{Zeq} satisfies all conditions of H\"ormander's theorem \cite[Theorem~1.3]{Hair11} (see also \cite[Theorem~6.1]{Pavl}).
	
We see that the drift $b^\eps$ is in $\CC^{\infty}$ and all its derivatives are bounded. Furthermore, using \eqref{lie}, we see that for $x\neq0$, $z\in\R$ we have $\spn(	\sigma,\bigl[\sigma,b^{\eps}(x,z)\bigr])=\R^2$, and for $x=0$, $z\in\R$ we have 
	$\spn(	\sigma,\Bigl[\sigma,\bigl[\sigma,b^\eps(x,z)\bigr]\Bigr])=\R^2$. Thus, the parabolic H\"ormander condition holds. Hence, all the conditions of the H\"ormander theorem are met and 
\cite[Theorem~1.3]{Hair11} implies that $(X^\eps,Z^\eps)$  is strong Feller. 
\end{proof}

Now we can show well-posedness of \eqref{eq11}--\eqref{eq22}.  
\begin{lemma}\label{L:31}
For any random vector $(\wh x_0,\wh z_0)$ independent of $\wh B$ the stochastic differential equation \eqref{eq11}--\eqref{eq22} has a unique strong solution with $(\wh X_0,\wh Z_0)=(\wh x_0,\wh z_0)$. This solution is a  Markov process in the state space $\R^2$ and its transition kernel $P_t$ is strong Feller for any $t>0$. 
%
\end{lemma}
\begin{proof}

First, we consider the case when the initial data 
$(\wh x_0,\wh z_0)$ is deterministic. Then it is immediate to see that for any $T>0$ a solution to \eqref{eq11}--\eqref{eq22}  satisfies 
\begin{equation}\label{ineqY}
\wh Z_t\ge \wh z_0 -T/2,
\end{equation}
$t\in[0,T]$. Hence, on time interval $[0,T]$, any solution to \eqref{eq11}--\eqref{eq22} solves SDE \eqref{eq1eps}--\eqref{eq2eps} with $(X_0^\eps,Z_0^\eps)=(\wh x_0,\wh z_0)$, $\eps=\exp(2 \wh z_0-T)$, $W=\wh B$ and vice versa.
Since, by \Cref{L:31aux}, the latter equation has a unique strong solution, we see that SDE \eqref{eq11}-\eqref{eq22}
has a unique strong solution on  $[0,T]$ and
\begin{equation}\label{ref}
(\wh X_t,\wh Z_t)=( X^\eps_t, Z^\eps_t),\quad t\in[0,T].
\end{equation}
Since $T$ is arbitrary, it follows that 
SDE \eqref{eq11}-\eqref{eq22} has a unique strong solution on $[0,\infty)$. 

Strong existence for the case of arbitrary initial data follows now from \cite[Theorem~1]{olav}, and strong uniqueness from \cite[Remark~IV.1.4]{Ikedawatanabe}. Moreover, \cite[Theorem~5.4.20]{KS} shows that $(\wh X_t,\wh Z_t)_{t\ge0}$ is a Markov process with the state space $\R^2$ equipped with the Borel topology.

Now let us show that $(P_t)_{t\ge0}$ is strong Feller. Let $f$ be an arbitrary bounded measurable function $\R^2\to\R$, let $(x_0,z_0)\in \R^2$. Let $(x_0^n,z_0^n)\in\R^2$, $n\in\Z_+$ be a sequence converging to $(x_0,z_0)$ as $n\to\infty$. Without loss of generality we can assume that $z_0^n\ge -2 |z_0|$ for all $n\in\Z_+$. Fix $t>0$. Then, denoting by $(P^{\eps}_t)_{t\ge0}$ the transition kernel associated with SDE \eqref{eq1eps}--\eqref{eq2eps}, we derive
\begin{equation}\label{SF1}
P_t f(x^n_0,z^n_0)=\E_{(x^n_0,z^n_0)}f (\wh X_t,\wh Z_t)=\E_{(x^n_0,z^n_0)}f ( X^\eps_t, Z^\eps_t)=P_t^\eps f(x^n_0,z^n_0),	
\end{equation}
where $\eps:=\exp(-4|z_0|-t)$ and we used here \eqref{ineqY} and \eqref{ref}.   By \Cref{L:31aux}, we have
\begin{equation}\label{SF2}
P_t^\eps f(x^n_0,z^n_0)\to P_t^\eps f(x_0,z_0)=P_t f(x_0,z_0),\quad \text{as $n\to\infty$}, 
\end{equation}
here we used once again  \eqref{ineqY} and \eqref{ref}. Combining \eqref{SF1} and \eqref{SF2}, we see that $P_t$ is strong Feller.
%
\end{proof}

To show uniqueness of the invariant measure of $(P_t)$, we will need the following support theorem. For $\delta>0$, $v\in\R^2$ let $B_{\delta,v}$ be the ball of radius $\delta$ centred at $v$.

\begin{lemma}\label{L:support}
For any $(x_0,z_0)\in\R^2$, $\delta>0$, there exists $T>0$ such that
\begin{equation*}
P_T((x_0,z_0), B_{\delta,(0,\log 2)})>0.
\end{equation*}
\end{lemma}
\begin{proof}
Fix $(x_0,z_0)\in\R^2$. Consider the following deterministic  control problem associated with \eqref{eq11}--\eqref{eq22}:
\begin{align}
\frac{d}{dt}x_t &= \bigl(-\frac12 x_t -\frac{2x_t}{x_t^2+e^{2z_t}}\bigr) + \sqrt{\kappa} \frac{d}{dt} U_t,\label{U1}\\
\frac{d}{dt}z_t &= \bigl(-\frac12  +\frac{2}{x_t^2+e^{2z_t}}\bigr),\label{U2}
\end{align}	
where $x(0)=x_0$, $z(0)=z_0$ and $U\in\CC^1([0,T];\R)$ is a non-random function with $U_0=0$. We claim that we can find $T>0$ and $U$ such that $x_T=0$ and $|z_T-\log 2|<\delta/2$.

First, we take a $\CC^1$ path $x\colon[0,1]\to\R$ such that $x(0)=x_0$, $x(1)=0$, $\frac{d}{dt}x(t)\Bigr|_{t=1}=0$. Let $z_t$, $t\in[0,1]$, be a solution to \eqref{U2} with  the initial condition $z_0$ (for $x$ constructed above). 

Consider now the equation  
\begin{equation*}
\frac{d}{dt}z_t = \bigl(-\frac12  +\frac{2}{e^{2z_t}}\bigr), \quad t\ge1    
\end{equation*}
with the initial condition $z_1$ constructed above. It is easy to see that there exists $T=T(x_0,z_0)>1$ such that $|z_T-\log 2|<\delta/2$. Set $x_t=0$ for $t\in[1,T]$.

Finally, let $U_t$, $t\in[0,T]$, be a $\CC^1$ path such that \eqref{U1} holds for $x,z$ constructed above and $U_0=0$. The desired control $U$ has been constructed.

Now for arbitrary $\eps>0$, consider the event
\begin{equation*}
A_\eps:=\{\sup_{t\in[0,T]}|W_t-U_t|<\eps\}.	
\end{equation*}
It is well-known (see, e.g., \cite[Theorem~38]{Freedman}) that $\P(A_\eps)>0$.
Let $(\wh X_t,\wh Z_t)_{t\in[0,T]}$ be the solution of \eqref{eq11}-\eqref{eq22} with the initial condition $(x_0,z_0)$. Then
\begin{equation}\label{glb}
z_t\ge z_0-T/2,\quad \wh Z_t\ge z_0-T/2,\qquad \text{for all $t\in[0,T]$}.
\end{equation} 
Therefore, for any $t\in[0,T]$ we have on $A_\eps$
\begin{equation}\label{almostall}
|\wh X_t-x_t|+|\wh Z_t-z_t|\le  C\int_0^t (|\wh X_s-x_s|+|\wh Z_s-z_s|)\,ds +\sqrt \kappa\eps,	
\end{equation}
where we used \eqref{glb} and the fact that the Lipschitz constant of the drift of SDE 
\eqref{eq11}-\eqref{eq22} is bounded on the set $\R\times [z_0-T/2,+\infty)$. By the Gronwall inequality and \eqref{almostall}, we have on $A_\eps$
\begin{equation*}
|\wh X_T-x_T|+|\wh Z_T-z_T|\le  C(T)\sqrt \kappa \eps.
\end{equation*}
Choose now $\eps$ small enough, such that the right-hand side of the above inequality is less than $\delta/2$. Then recalling that $x_T=0$ and 
$|z_T-\log 2|<\delta/2$, we finally deduce 
\begin{equation*}
P_T((x_0,z_0), B_{\delta,(0,\log 2)})\ge \P(A_\eps)>0.\qedhere
\end{equation*}
\end{proof}

\begin{lemma}\label{L:UIM}
The measure $\pi:=\Law\bigl(\Re(\gamma_1),\log(\Im(\gamma_1))\bigr)$ is the unique invariant measure for the process \eqref{eq11}--\eqref{eq22}.
\end{lemma}

\begin{proof} 
The fact that the measure $\pi$ is invariant follows by a standard argument. Denote, as usual,
for a measurable bounded function $f\colon\R^2\to\R$ and a measure $\nu$ on $\R^2$
\begin{align*}
P_tf(x):=\int_{\R^2} f(y) P_t(x,dy),\,\,x\in\R^2;\qquad
P_t\nu(A):=\int_{\R^2} P_t(y,A)\,\nu(dy),\,\,A\in\mathcal{B}(\R^2).
\end{align*}
Consider the measure $\mu:=\Law\Bigl(\Re (h_1(i))-\sqrt\kappa \wt B_1,
\log(\Im (h_1(i)))\Bigr)$. Rewriting \eqref{corres}, we see that
\begin{equation}\label{convergence2}
P_t \mu\to \pi\,\,\text{weakly as $t\to\infty$}.
\end{equation}
Fix any $s\ge0$. Let us show that $P_s \pi=\pi$. Indeed, let $f\colon\R^2\to\R$ be an arbitrary continuous bounded function. Then
\begin{align*}
	\int_{\R^2} f(x) \,P_s\pi(dx)&=
	\int_{\R^2} P_sf(x) \,\pi(dx)
	=\lim_{t\to\infty}\int_{\R^2} P_sf(x) \,P_t\mu(dx)\\
	&=\lim_{t\to\infty}\int_{\R^2} f(x) \,P_{t+s}\mu(dx)=\int_{\R^2} f(x) \,\pi(dx),
\end{align*}
where the second identity follows from \eqref{convergence2} and  the fact that $P_sf$ is a bounded continuous function (this is guaranteed by the Feller property of $P$). Since $f$ was arbitrary bounded continuous function, we see that $P_s \pi=\pi$ for any $s\ge0$. Thus, the measure $\pi$ is invariant for  SDE \eqref{eq11}-\eqref{eq22}.
	
Now let us show that  SDE \eqref{eq11}-\eqref{eq22} have a unique invariant measure. Assume the contrary. Then SDE \eqref{eq11}-\eqref{eq22} must have two different ergodic invariant measures $\nu$, $\wt \nu$  (\cite[Lemma~7.1]{HaiPDE}, \cite[Theorem~5.1.3(iv)]{VO16}). By \Cref{L:31} the semigroup $(P_t)$ is strong Feller. Therefore, by \cite[Proposition 7.8]{DPG}
\begin{equation}\label{support}
\supp(\nu)\cap\supp(\wt \nu)=\emptyset.	
\end{equation} 
We claim now that the point $(0,\log2)$ belongs to the support of both of these measures.

Indeed, fix arbitrary $\delta>0$. Take any $(x_0,z_0)\in\supp(\nu)$. Then, by \Cref{L:support}, there exists $T>0$, $\eps>0$ such that $P_T((x_0,z_0), B_{\delta,(0,\log 2)})>\eps$. By the strong Feller property of $P_T$, the function $(x,z)\mapsto P_T((x,z), B_{\delta,(0,\log 2)})$ is continuous. Therefore, there exists $\delta'>0$ such that
\begin{equation*}
P_T((x,z), B_{\delta,(0,\log 2)})>\eps/2,\quad \text{for any $(x,z)\in
	B_{\delta',(x_0,z_0)}$}. 
\end{equation*}
This implies that 
\begin{equation*}
\nu(B_{\delta,(0,\log 2)})\ge\int_{B_{\delta',(x_0,z_0)}} \nu(x,z)	P_T((x,z), B_{\delta,(0,\log 2)})\,dxdz\ge \frac\eps2\nu (B_{\delta',(x_0,z_0)})>0 
\end{equation*}
where the last inequality follows from the fact that $(x_0,z_0)\in\supp(\nu)$.  Since $\delta$ was arbitrary, we see that  
$(0,\log 2)\in\supp(\nu)$. Similarly, $(0,\log 2)\in\supp(\wt\nu)$, which contradicts \eqref{support}. Therefore,  SDE \eqref{eq11}-\eqref{eq22} has a unique invariant measure.  
\end{proof}	
	
Let $L$ be the generator of the semigroup $P$
\begin{equation*}
L f := \frac12\kappa \d^2_{xx}f +\bigl(-\frac12 x - \frac{2x}{x^2+e^{2z}}\bigr)\d_xf+\bigl(-\frac12  + \frac{2}{x^2+e^{2z}}\bigr)\d_zf,
\end{equation*}
where $f\in\CC^\infty(\R^2)$. As usual, the adjoint of $L$ will be denoted by $L^*$. 

\begin{lemma}\label{L:36}
The measure $\pi:=\Law\bigl(\Re(\gamma_1),\log(\Im(\gamma_1))\bigr)$ has a smooth density $p$ with respect to the Lebesgue measure. Further, $p$ is the unique solution  in the class of densities of the Fokker-Planck-Kolmogorov equation
\begin{equation}\label{KFP}
L^*p=0.
\end{equation}
Finally, $p(x,z)=0$ for $x\in\R$, $z\ge\log 2$, and $p(x,z)>0$
 for $x\in\R$, $z<\log 2$.   
\end{lemma}
\begin{proof}
Since the measure $\pi$ is invariant for $P$, we have (in the weak sense)
\begin{equation}\label{stars}
	L^*\pi=0.
\end{equation}
Let us now check that $L^*$ satisfies the (standard)  H\"ormander condition.

Denote by $b$ the drift of \eqref{eq11}-\eqref{eq22}
\begin{equation}\label{driftdefb}
	b(x,z):=\begin{pmatrix}
		b^{1}(x,z)\\b^{2}(x,z) 
	\end{pmatrix}:=\begin{pmatrix}
		-\frac12 x -\frac{2x}{x^2+e^{2z}}\\[1.5ex]
		-\frac12  +\frac{2}{x^2+e^{2z}} 
	\end{pmatrix},\quad x,z\in\R.
\end{equation}
and recall the notation for $\sigma$ \eqref{bepss}. Using \eqref{lie}, we see that  for $x\neq0$, $z\in\R$ we have $\spn(	\sigma,\bigl[\sigma, b(x,z)\bigr])=\R^2$, and for $x=0$, $z\in\R$ we have 
	$\spn(	\sigma,\Bigl[\sigma,\bigl[\sigma, b(x,z)\bigr]\bigr])=\R^2$. Thus, the  H\"ormander condition holds and by   
H\"ormander's theorem  \cite[Theorem~7.4.3]{StPDE}, $L^*$ is hypoelliptic.\footnote{In the proof of \cref{L:31aux}, we use \cite[Theorem~1.3]{Hair11} which is a probabilistic version of H\"ormander's theorem, and it imposes global assumptions on boundedness of derivatives of the drift. Here we use  \cite[Theorem~7.4.3]{StPDE} which is a purely PDE result and it does not require any global assumptions. Therefore we do not have to smoothen the drift $b$ here.} Therefore, \eqref{stars} implies that the Schwarz distribution $\pi\in\CC^\infty(\R^2)$. Thus, the measure $\pi$ has a $\CC^\infty$ density $p$ with respect to the Lebesgue measure and \eqref{KFP} holds.

Now let us show that \eqref{KFP} does not have any other solutions. We have already seen that semigroup $(P_t)$  has a unique invariant measure (this has been established in \cref{L:UIM}).  
In general, without extra conditions, this does not immediately imply uniqueness of solutions to \eqref{KFP} in the class of probability measures, see \cite[hint to exercise~9.8.48]{BKRS}. This is because not every probability solution to the Fokker--Planck--Kolmogorov equation corresponds to a solution of the martingale problem; we refer to \cite[p.~719]{BRS} for further discussion.

Thus, we assume the contrary and suppose that $p'$ is another probability density which solves  \eqref{KFP}. Let $\pi'$ be the measure with density $p'$. We claim that $\pi'$ is another invariant measure for $(P_t)$. 

Consider a Lyapunov function $V$ (the suggestion to take this specific function is due to Stas Shaposhnikov)
\begin{equation*}
V(x,z):=x^2 + \log (1+z^2),\quad (x,z)\in\R^2.
\end{equation*}
Then 
\begin{align*}
LV(x,z)&=\kappa-x^2 -\frac{4x^2}{x^2+e^{2z}}-\frac{z}{1+z^2}+ \frac{4z}{(x^2+e^{2z})(1+z^2)}\\
&\le \kappa +3-\Bigl(x^2 + \frac{4|z|\ind(z\le0)}{(x^2+e^{2z})(1+z^2)}\Bigr).
\end{align*}
By \cite[Theorem~2.3.2 and inequality (2.3.2)]{BKRS}, this implies (note that $V$ is obviously quasi-compact in the sense of 
\cite[Definition~2.3.1]{BKRS})
\begin{equation}\label{almostdone}
\int_{\R^2}\bigl(x^2 + \frac{4|z|\ind(z\le0)}{(x^2+e^{2z})(1+z^2)}\bigr)p'(x,z)\,dxdz<\infty.
\end{equation}

Then recalling \eqref{driftdefb} we have
\begin{equation*}
\frac{1+|b^1(x,z)x|+|b^2(x,z)z|}{1+x^2+z^2}\le 5+ \frac{2|z|\ind(z\le0)}{(x^2+e^{2z})(1+z^2)}.
\end{equation*}
Combining this with \eqref{almostdone}, we see that  for any $T>0$
\begin{equation*}
\int_0^T\int_{\R^2}	\frac{1+|b^1(x,z)x|+|b^2(x,z)z|}{1+x^2+z^2}p'(x,z)\,dxdzdt<\infty.
\end{equation*}
Therefore, by the generalized Ambrosio-Figalli-Trevisan superposition principle \cite[Theorem~1.1]{BRS} and the standard equivalence between weak solutions of SDE and the martingale problems, see, e.g., \cite[Proposition~5.4.11]{KS}, there exists a weak solution to SDE \eqref{eq11}--\eqref{eq22} on the interval $[0,T]$ such that for any $t\ge0$ we have $\Law(\wh X_t, \wh Z_t)=\pi'$. Thus, the measure $\pi'$ is also invariant for the semigroup $(P_t)$. However, this contradicts \Cref{L:UIM}. Therefore, \eqref{KFP} has a unique solution in the class of probability densities.

Finally, let us prove the results concerning the support of $p$. 
 Note that if $\wh Z_0(\omega)>\log 2$, then $\wh Z_0(\omega)>\wh Z_1(\omega)$. Let $f\colon\R\to[0,\infty)$ be an increasing function such that $f(x)=0$ for $x\le \log2$ and $f(x)>0$ for $x>\log 2$. Then $f(\wh Z_0)-f(\wh Z_1)\ge0$. On the other hand, by invariance
\begin{equation*}
\E_\pi (f(\wh Z_0)-f(\wh Z_1))=0.
\end{equation*}  
This implies that $\P_\pi$ a.s. we have $f(\wh Z_0)=f(\wh Z_1)$. By the definition of $f$ this implies that $\P_\pi(\wh Z_0>\log2)=0$ and thus $\pi(\R\times (\log2,\infty))=0$. Since the density $p$ is continuous we have 
\begin{equation}\label{topsup}
p(x,z)=0,\quad x\in\R, z\ge\log2.	
\end{equation}

\begin{figure}[h]
\centering
\includegraphics[width=0.8\textwidth]{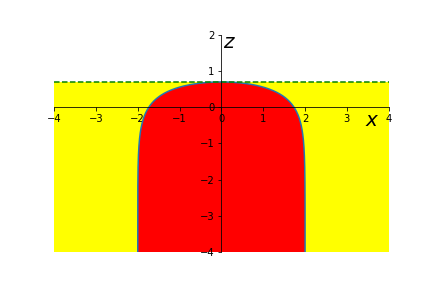}
\caption{Support of the density $p$ (yellow and red regions). The process $\wh Z_t$ is increasing when $(\wh X_t,\wh Z_t)$ is in the red region, and decreasing whenever $(\wh X_t,\wh Z_t)$ is in the yellow region. The dashed line, which touches the red region, is $z=\log 2$.}
\label{fig:supp}
\end{figure}

Now let us show that $p(x,z)>0$ for any $z<\log2$. 
The idea of this part of the proof is due to Stas Shaposhnikov. Suppose the contrary that for some $x_0\in\R$, $z_0<\log 2$ we have $p(x_0,z_0)=0$. We claim that this implies that $p\equiv0$. Note that  the set $\{z=z_0\}$ is the set of elliptic connectivity for operator $L^*$ in the sense of \cite[Chapter III.1]{OR73} (see also \cite[Section 2]{Hill}). Therefore, the maximum principle for degenerate
elliptic equations \cite[Theorem 3.1.2]{OR73} (see also \cite[Theorem~1]{Hill}, \cite[Theorem~4]{Aleks}) implies that $p(x,z_0)=0$ for any $x\in\R$.

Note that in the domain 
$$
D:=\{x^2+\exp(2z)<4\}
$$ 
PDE \eqref{KFP} becomes a parabolic equation in $(z,x)$ and on its complement \eqref{KFP} is a backward parabolic equation. This corresponds to the fact that the process $\wh Z_t$ is increasing on $D$ and decreasing on $\R^2\setminus D$, see \cref{fig:supp}.

Fix now small $\delta$ such that $\delta^2 + \exp(2 z_0)<4$ (this is possible since $z_0<\log2$). Consider now the domain $D':=[-\delta,\delta]\times (-\infty,z_0)\subset D$.  In this domain \eqref{KFP} is a parabolic equation
\begin{equation}\label{parpde}
\d_z p -a(x,z)\d_{xx} p+b(x,z)\d_x p +c (x,z) p =0,
\end{equation} 
for certain smooth functions $a,b,c$ and
$$
a(x,z)=\frac{\kappa}{\frac{4}{x^2+e^{2z}}-1}>0, \quad (x,z)\in D',
$$  
since $\frac{4}{x^2+e^{2z}}>\frac{4}{\delta^2+e^{2z_0}}>1$ on $D'$.
Therefore, by the Harnack inequality for parabolic equations (see, e.g., \cite[Section~7.1, Theorem~10]{Evans}), we get for arbitrary $z_1\le z_0$, and $C>0$
\begin{equation*}
\sup_{x\in(-\delta,\delta)} p(x, z_1) \le  C \inf_{x\in(-\delta,\delta)} p(x, z_0)=0.	
\end{equation*}
Using again the maximum principle for degenerate
elliptic equations, we deduce from this that $p(x,z_1)=0$ for any $x\in\R$. Since $z_1\le z_0$ was arbitrary we have that $p\equiv 0$ on $\R\times(-\infty;z_0]$. 

We use a similar argument to treat the case $z\ge z_0$. Consider now the domain $D'':=[3,4]\times (z_0,\infty)\subset \R^2\setminus D$.  In this domain \eqref{KFP} is a \textit{backward} parabolic equation \eqref{parpde} and 
$$
a(x,z)=\frac{\kappa}{\frac{4}{x^2+e^{2z}}-1}<0, \quad (x,z)\in D''
$$ 
since $\frac{4}{x^2+e^{2z}}<\frac49<1$ on $D''$.
The Harnack inequality for parabolic equations implies now for arbitrary $z_1\ge z_0$, and $C>0$
\begin{equation*}
	\sup_{x\in(3,4)} p(x, z_1) \le  C \inf_{x\in(3,4)} p(x, z_0)=0.
\end{equation*}
and thus, as above, the maximum principle implies that $p\equiv 0$ on $\R\times[z_0, \infty)$. 

Therefore the function $p$ is identically $0$ which is not possible since $p$ is a density. This contradiction shows that $p(x,z)>0$ for any $x\in\R$, $z<\log 2$. Together with \eqref{topsup} this concludes the proof of the theorem.
\end{proof}

\begin{proof}[Proof of \Cref{T:main}]
By \Cref{L:36}, the measure $\Law\bigl(\Re(\gamma_1),\log(\Im(\gamma_1))\bigr)$ has a smooth density $p$ with respect to the Lebesgue measure, which solves \eqref{KFP}. Therefore, the measure $\Law\bigl(\Re(\gamma_1),\Im(\gamma_1)\bigr)$
has a density 
$$
\psi(x,y):=\frac1y p(x,\log y),\quad x\in\R, y>0.
$$
Now, by change of variables, it is easy to see that $\psi$ is the unique solution of \eqref{mainPDE} in the class of probability densities. Since $p(x,z)$ is positive whenever $z<\log 2$, we see that $\psi(x,y)$ is positive whenever $y\in(0,2)$. Finally, it is immediate that the function $\bar\psi(x,y):=\psi(-x,y)$ also solves \eqref{mainPDE}. By uniqueness, this implies that $\psi(x,y)=\psi(-x,y)$.
\end{proof}

\subsection{Proof of \cref{T:other}}
\label{se:density_analysis}

To establish \cref{T:other}, it will be convenient to work in the coordinates $(A,U)$, where 
\begin{equation*}
A:=\arg\gamma_1 = \cot^{-1}(\Re\gamma_1/\Im\gamma_1), \quad U:=(\Im\gamma_1)^2.
\end{equation*}

Denoting the density of $(A,U)$ by $\phi$, we note that
\[
\psi(x,y)=\frac{2 y^2}{x^2+y^2}\phi(\cot^{-1}(x/y),y^2),\quad x\in\R,\ y>0.
\]
It follows from \Cref{T:main} that the density $\phi$ is the unique solution to the corresponding Fokker-Planck-Kolmogorov equation, which in the new coordinates is given by 
\begin{align}\label{FPK}
&\frac{\kappa}{2u}\sin^4\alpha\, \d^2_{\alpha\alpha}\phi +\frac{3\kappa-4}{u}\sin^3\alpha\cos\alpha\, \d_\alpha\phi\nn +(u-4\sin^2\alpha)\, \d_u\phi \\&\qquad+\frac{\kappa-4}{u}(3\sin^2\alpha\cos^2\alpha-\sin^4\alpha)\, \phi + \phi=0,\quad (\alpha,u)\in(0,\pi)\times(0,4].
\end{align}
Recall that we can consider this equation on a larger domain $(0,\pi)\times(0,\infty)$, but since $\psi(x,y)=0$ for $y\ge2$, we have $\phi(\alpha,u)=0$ for $u\ge4$.

Note that this PDE can be rewritten as
\begin{equation}\label{eq:fp_ualpha}
\partial_u((u-4\sin^2\alpha)\, \phi)+\frac{\kappa-4}{u}\partial_\alpha(\sin^3\alpha\cos\alpha \,\phi)+\frac{\kappa}{2u}\partial_\alpha(\sin^4\alpha\,\partial_\alpha\phi)=0.
\end{equation}
The crucial statement on the way to prove \cref{T:other} is the following lemma.
\begin{lemma}\label{L:db1}
For any $\alpha\in(0,\pi)$ we have
\begin{equation}\label{mainidentity}
\int_0^4 \frac{1}{u}\phi(\alpha,u)\,du = \frac{\Gamma(1+\frac4\kappa)}{4\sqrt\pi\Gamma(\frac12+\frac4\kappa)}(\sin\alpha)^{8/\kappa-2}.	
\end{equation}
\end{lemma}

Before we go into the technical details, let us outline heuristically the main idea of the proof. If we assume $\phi(\alpha,0+)=\phi(\alpha,4)=0$, then integrating \eqref{eq:fp_ualpha} in $u$ yields
\begin{equation*}
\partial_\alpha \Bigl(\int_0^4 (\frac{\kappa-4}{u}\sin^3\alpha\cos\alpha\,\phi(\alpha,u) +\frac{\kappa}{2u}\sin^4\alpha\,\d_\alpha \phi(\alpha,u))\,du \Bigr)=0.
\end{equation*}
Hence the expression $J(\alpha):=\int_0^4( \frac{\kappa-4}{u}\sin^3\alpha\cos\alpha\,\phi +\frac{\kappa}{2u}\sin^4\alpha\,\d_\alpha \phi)\,du$ does not depend on $\alpha$. Moreover, let us suppose that $\alpha^4 |\d_\alpha\phi|$ and $\alpha^3\phi$ monotonically go to $0$ as $\alpha\to0$ for any $u\in(0,4]$. Then $J(0+)=0$ and thus $J(\alpha)=0$ for any $\alpha\in(0,\pi)$. Therefore,
\begin{equation*}
0=J(\alpha)\sin^{-8/\kappa-2}\alpha=\int_0^4 \frac{\kappa}{2u}\partial_\alpha((\sin\alpha)^{2-8/\kappa}\,  \phi(\alpha,u))\, du.
\end{equation*}
This yields that $\int_0^4 \frac{1}{u}(\sin\alpha)^{2-8/\kappa}\,  \phi(\alpha,u)\, du$ is constant in $\alpha$, which gives
\begin{equation*}
	\int_0^4 \frac{1}{u}\, \phi(\alpha,u)\, du = c(\sin\alpha)^{8/\kappa-2}
\end{equation*}
for some $c>0$, which is almost the statement of \Cref{L:db1}.

However, since the  boundary behavior of $\phi$ as $\alpha$ approaches $0$ is not  clear, we developed an alternative approach which avoids these steps. Instead of integrating all the way to $0$, we will integrate only up to $\varepsilon > 0$ and obtain approximate identities. Then we would like to let $\varepsilon \searrow 0$. For this, we would need  the following technical results about approximating ODEs.

\begin{lemma}\label{le:ode_approximation}
Let $S,T\in\R$, $S\le T$. Suppose $x_k\colon [S,T] \to \RR^d$, $k \in \NN$, are continuous functions that solve the integral equation
\begin{equation}\label{inteqn}
x_k(t)-x_k(s) = \int_{s}^{t} \big(F(r,x_k(r)) + g_k(r)\big) \,dr +h_k(s,t),\quad s,t\in[S,T],
\end{equation}
where
\renewcommand\labelitemi{$\vcenter{\hbox{\tiny$\bullet$}}$}
\begin{itemize}
    \item $F$ is a continuous function $[S,T]\times \R^d\to\R^d$ and there exists $C>0$ such that $|F(t,x)| \le C(1+|x|)$ for $t \in [S,T]$, $x\in \R^d$;
    \item $g$ and $g_k$, $k\in\Z_+$, are integrable functions $[S,T]\to\R^d$, $g_k\to g$ pointwise as $k\to\infty$, and $\sup_k \norm{g_k}_\infty < \infty$;
    \item $h_k$, $k\in\Z_+$, are functions $[S,T]^2\to\R^d$, and $\norm{h_k}_\infty \to 0$ as $k \to \infty$.
\end{itemize}
Moreover, suppose that there exist $t_k \in [S,T]$ such that $\sup_k |x_k(t_k)| < \infty$.

Then there exists a continuous  function $x\colon [S,T] \to \R^d$ such that along some subsequence $(k_j)_{j\in\Z_+}$ we have $x_{k_j} \to x$ uniformly as $j\to\infty$ and
\begin{equation}\label{inteq}
x(t)-x(s) = \int_{s}^{t} \big(F(r,x(r)) + g(r)\big) \,dr,\quad s,t\in[S,T].
\end{equation}
\end{lemma}

\begin{proof}
First, we show that $x_k$ are uniformly bounded. Indeed, by our assumptions we have
for any $t\in[S,T]$
\[ \begin{split}
    \abs{x_k(t)} &\le \abs{x_k(t_k)} + \int_{t_k}^t \abs{F(r,x_k(r)) + g_k(r)} \,dr + \norm{h_k}_\infty \\
    &\le C + C \int_{t_k}^t (1+\abs{x_k(r)})\,dr ,
\end{split} \]
and an application of Grönwall's inequality implies $x_k$ are uniformly bounded.

Consequently, we can assume $F$ to be bounded. It follows that the family $(x_k)_{k\in\Z_+}$ is equicontinuous. Indeed, for $\varepsilon > 0$ let $k_\varepsilon$ large enough such that $\norm{h_k}_\infty < \varepsilon$ for $k \ge k_\varepsilon$. Then, for $k \ge k_\varepsilon$, we have
\[ \begin{split}
    \abs{x_k(t)-x_k(s)} &\le \int_s^t \abs{F(r,x_k(r)) + g_k(r)} \, dr + \varepsilon \\
    &\le C\abs{t-s} + \varepsilon
\end{split} \]
which is smaller than $2\varepsilon$ whenever $\abs{t-s} < \varepsilon/C$. For $k < k_\varepsilon$, by continuity of $x_k$ we can find $\delta_\varepsilon > 0$ such that $\abs{x_k(t)-x_k(s)} < \varepsilon$ whenever $\abs{t-s} < \delta_\varepsilon$.

Hence, by the Arzelà-Ascoli theorem, we have $x_{k_j} \to x$ uniformly along some subsequence. Equation \eqref{inteq} follows now from \eqref{inteqn} by taking limits.
\end{proof}

We will later also frequently use integation by parts arguments. In order to control the boundary terms that appear, the following lemma will be useful.

\begin{lemma}\label{le:nice_subseq}
Let $T>0$, and $f\colon (0,T] \to \RR$ be a differentiable function such that $\int_\varepsilon^T f(s)\,ds$ neither diverges to $+\infty$ nor $-\infty$ as $\varepsilon \searrow 0$. Let $h\colon {(0,T]} \to {(0,\infty)}$ be a non-increasing differentiable function such that $\int_0^T h(s)\,ds = +\infty$. Then there exists a sequence $t_k \searrow 0$ such that
\[
    \abs{f(t_k)} \le h(t_k) \quad\text{and}\quad f'(t_k) \ge h'(t_k),\quad k\in\Z_+.
\]
\end{lemma}

\begin{proof}
    First we note that there must exist a sequence $s_k \searrow 0$ such that $\abs{f(s_k)} < h(s_k)$ for all $k\in\Z_+$, otherwise we would have $\int_0^T f(s)\,ds =+ \infty$ or $\int_0^T f(s)\,ds = -\infty$. To control $f'$, we distinguish two cases.
    
    \textbf{Case 1:} We have $\abs{f(t)} \le h(t)$ for all small $t$. In that case, consider $g(t) := h(t)-f(t)$. The function $g$ cannot be always increasing for small $t$, otherwise we would have $\int_0^T f(s)\,ds = \infty$. Consequently there must be a sequence $t_k \searrow 0$ such that $g'(t_k) \le 0$.
    
    \textbf{Case 2:} We have $\abs{f(r_k)} > h(r_k)$ along a sequence $r_k \searrow 0$. We can pick the sequence such that either $f(r_k) > h(r_k)$ for all $k$ or $f(r_k) < -h(r_k)$ for all $k$. In the former case $f(r_k) > h(r_k)$ for all $k$, let $t_k = \sup\{ t < r_k \mid f(t) \le h(t) \}$ (this set is non-empty due to the existence of a sequence $s_k$ with $\abs{f(s_k)} < h(s_k)$). Then $f(t_k) = h(t_k)$ and $f'(t_k) \ge h'(t_k)$ as desired. In the latter case $f(r_k) < -h(r_k)$ for all $k$, let $t_k = \inf\{ t > r_k \mid f(t) \ge -h(t) \}$ (again, $(t_k)$ is well-defined and tends to $0$ due to the existence of $(s_k)$ as above). Then $f(t_k) = h(t_k)$ and $f'(t_k) \ge -h'(t_k) \ge h'(t_k)$ since $h$ is non-increasing.
\end{proof}

\begin{corollary}\label{Cor:310}
Consider the same setup as \cref{le:nice_subseq}, and suppose additionally that $f \ge 0$. 
Then there exists a sequence $t_k \searrow 0$ such that
\begin{equation}\label{maincondtk}
f(t_k)\le h(t_k) \quad\text{and}\quad \abs{f'(t_k)} \le \abs{h'(t_k)},\quad k\in\Z_+.
\end{equation}
\end{corollary}

\begin{proof}
    Let $(t_k)_{k\in\Z_+}$ be a sequence as in \cref{le:nice_subseq}. Fix now $k\in\Z_+$. If $f'(t_k) \le \abs{h'(t_k)}$, then, by \cref{le:nice_subseq} we have  $f'(t_k) \ge -\abs{h'(t_k)}$ and $f(t_k)\le h(t_k)$. Hence the point $t_k$ satisfies \eqref{maincondtk}. 
    
   Otherwise, if  $f'(t_k) > \abs{h'(t_k)}$, define $s_k = \sup\{ t \le t_k \mid f'(t) \le \abs{h'(t)} \}$ (this set is non-empty, otherwise we would have $f(t) \to -\infty$ as $t \searrow 0$). By definition, we have $f'(t) > \abs{h'(t)}$ for $t \in {]s_k, t_k]}$, and hence also $f(s_k) < f(t_k)$. Moreover, we find some $r_k \le s_k$ close to $s_k$ with $f'(r_k) \le \abs{h'(r_k)}$. By continuity, we still have $f(t) < f(t_k)$ for $t \in [r_k,t_k[$. Since the derivative of any differentiable function satisfies the intermediate value theorem, we find some $\wt t_k \in [r_k,t_k]$ such that $f'(\wt t_k) = \abs{h'(\wt t_k)}$.   
   Then we also have $f(\wt t_k) < f(t_k) \le h(t_k) \le h(\wt t_k)$ as desired.
\end{proof}

We now proceed to the main part of our proof. In the following, we denote for $n\in\Z_+$, $\alpha\in(0,\pi)$, and $\eps>0$
\begin{equation*}
I_n(\alpha):= \int_0^4 u^{-n} \phi(\alpha,u)\, du;\qquad
I_n^{\eps}(\alpha): = \int_\eps^4 u^{-n} \phi(\alpha,u)\, du.
\end{equation*}
From the equation \eqref{FPK}, we will deduce a recursive system of ODEs that are satisfied for the functions $I_n$. In fact, the relation is satisfied for general $n \in \RR$ but we will use it only with $n\in\Z_+$.

\begin{lemma}\label{L:mainlemma} Let $n\in\Z_+$ be fixed. Suppose that either $n=0$ or $I_n$ is continuous (and finite) on  $(0,\pi)$. Assume that for any $\delta>0$ there exists a sequence $(\eps_k)_{k\in\Z_+}$ converging to $0$ such that 
\begin{equation}\label{formeraQ}
\eps_k^{-n}\int_\delta^{\pi-\delta}  \phi(\alpha,\eps_k)\,d\alpha \to 0
\end{equation}
Then either $I_{n+1} = \infty$ everywhere on $(0,\pi)$ or $I_{n+1}$ is twice differentiable on $(0,\pi)$ and satisfies the following ODE
\begin{multline}\label{eq:Iq_ode}
		0 = nI_n - 4n\sin^2\alpha \,I_{n+1}
		+(\kappa-4)(3\sin^2\alpha\cos^2\alpha-\sin^4\alpha) \,I_{n+1} \\
		+(3\kappa-4)\sin^3\alpha\cos\alpha \,I_{n+1}'
		+\frac{\kappa}{2}\sin^4\alpha \,I_{n+1}'' .
	\end{multline}
\end{lemma}
\begin{proof}
Fix $n\in\Z_+$. Let  $\eps>0$. Multiplying \eqref{FPK} by $u^{-n}$ and integrating in $u \in [\varepsilon,4]$ yields
\begin{align}
0 ={}& \int_\varepsilon^4 u^{-n} \partial_u((u-4\sin^2\alpha)\, \phi) \, du\nn\\
		&+(\kappa-4)(3\sin^2\alpha\cos^2\alpha-\sin^4\alpha) \int_\varepsilon^4 u^{-n-1}\, \phi \, du \nn\\
		&+(3\kappa-4)\sin^3\alpha\cos\alpha \int_\varepsilon^4 u^{-n-1}\, \d_\alpha\phi \, du
		+\frac{\kappa}{2}\sin^4\alpha \int_\varepsilon^4 u^{-n-1}\, \d^2_{\alpha\alpha}\phi \, du \nn\\
		={}& -\eps^{-n} (\eps-4\sin^2\alpha)\, \phi(\alpha,\eps) +nI_n^{\varepsilon} - 4n\sin^2\alpha \,I_{n+1}^{\varepsilon}\nn\\
		&+(\kappa-4)(3\sin^2\alpha\cos^2\alpha-\sin^4\alpha) \,I_{n+1}^{\varepsilon} \nn\\
		&+(3\kappa-4)\sin^3\alpha\cos\alpha \,(I_{n+1}^{\varepsilon})'
		+\frac{\kappa}{2}\sin^4\alpha \,(I_{n+1}^{\varepsilon})'' .\label{prelimpde}
	\end{align}
We would like to apply \cref{le:ode_approximation} to pass to the limit as $\eps\to0$ in the above ODE. Suppose now that $I_{n+1}$ is not infinite everywhere, i.e. $I_{n+1}(\alpha_0) < \infty$ for some $\alpha_0 \in (0,\pi)$. Fix arbitrary $\delta>0$ small enough such that $\alpha_0 \in (\delta,\pi-\delta)$ and set $S\defeq\delta$, $T\defeq\pi-\delta$,
\begin{align*}
x_k &\defeq \begin{pmatrix}
           I_{n+1}^{(\eps_k)} \\
           \d_\alpha I_{n+1}^{(\eps_k)}
         \end{pmatrix},\\
F(\alpha,x^{(1)},x^{(2)}) &\defeq \begin{pmatrix}
        x^{(2)}\\[1ex]
        \frac{2}{\kappa\sin^2 \alpha}\left( 4n x^{(1)}-x^{(1)}(\kappa-4)(3\cos^2\alpha-\sin^2\alpha) 
		-(3\kappa-4)x^{(2)}\sin\alpha\cos\alpha  \right)
         \end{pmatrix},\\
g_k(\alpha) &\defeq \begin{pmatrix} 0 \\ -\frac{2}{\kappa\sin^4 \alpha} nI_n^{\eps_k}(\alpha) \end{pmatrix},\\
h_k(\alpha_1,\alpha_2) &\defeq \begin{pmatrix} 0 \\  \int_{\alpha_1}^{\alpha_2}\frac{2}{\kappa\sin^4 \alpha} \eps_k^{-n} (\eps_k-4\sin^2\alpha)\, \phi(\alpha,\eps_k)\,d\alpha \end{pmatrix}
\end{align*}
where $\eps_k$ are the same as in the condition \eqref{formeraQ}. 
It is obvious that on $[\delta, \pi-\delta]$ the function $F$ is continuous 
and has linear growth in $x^{(1)}, x^{(2)}$. Moreover, $g_k(\alpha) \to g(\alpha) = \begin{pmatrix} 0 & -\frac{2nI_n(\alpha)}{\kappa\sin^4 \alpha} \end{pmatrix}^T$ monotonically by the assumptions of the Lemma. Finally, thanks to \eqref{formeraQ}, we have $\norm{h_k}_\infty \to 0$ on $[\delta, \pi-\delta]^2$.

It remains to find a sequence $\alpha_k$ such that $I_{n+1}^{(\eps_k)}(\alpha_k)$ and $(I_{n+1}^{(\eps_k)})'(\alpha_k)$ are bounded.

First assume that there exists $\alpha',\alpha''\in[\delta,\pi-\delta]$ such that $\alpha' < \alpha_0 < \alpha''$ and $I_{n+1}(\alpha_0)< I_{n+1}(\alpha')$, $I_{n+1}(\alpha_0)< I_{n+1}(\alpha'')$ (at this point,  we allow $I_{n+1}(\alpha')$ or  $I_{n+1}(\alpha'')$ to be infinite). Then for all large enough $k$ we have $I_{n+1}^{\eps_k}(\alpha_0)< I_{n+1}^{\eps_k}(\alpha')$, $I_{n+1}^{\eps_k}(\alpha_0)< I_{n+1}^{\eps_k}(\alpha'')$. Pick some $\alpha_k \in \operatorname{argmin}_{[\alpha',\alpha'']} I_{n+1}^{\eps_k}$. By above, $\alpha_k\in(\alpha',\alpha'')$ and hence $(I_{n+1}^{\varepsilon_k})'(\alpha_k) = 0$. Moreover, $I_{n+1}^{\varepsilon_k}(\alpha_k) \le I_{n+1}^{\varepsilon_k}(\alpha_0)\le I_{n+1}(\alpha_0)$. Thus the sequence $\bigl(I_{n+1}^{\varepsilon_k}(\alpha_k),(I_{n+1}^{\varepsilon_k})'(\alpha_k)\bigr)_{k\in\Z_+}$ is bounded. 

If $I_{n+1}(\alpha_0)\ge I_{n+1}(\alpha)$ for all $\alpha\in[\delta,\alpha_0]$, then 
\begin{equation}\label{case1proof}
\sup_{\alpha\in[\delta,\alpha_0]}I_{n+1}^{\varepsilon_k}(\alpha)\le 
\sup_{\alpha\in[\delta,\alpha_0]}I_{n+1}(\alpha)\le I_{n+1}(\alpha_0).
\end{equation}
Hence for each $k$ there exists $\alpha_k\in[\delta,\alpha_0]$ such that $\abs{(I_{n+1}^{\varepsilon_k})'(\alpha_k)}\le I_{n+1}(\alpha_0)/(\alpha_0-\delta)$. Combining this with \eqref{case1proof} we see again that the sequence $\bigl(I_{n+1}^{\varepsilon_k}(\alpha_k),(I_{n+1}^{\varepsilon_k})'(\alpha_k)\bigr)_{k\in\Z_+}$ is bounded.

The case when $I_{n+1}(\alpha_0)\ge I_{n+1}(\alpha)$ for all $\alpha\in[\alpha_0,\pi-\delta]$ is treated in a similar way.

Thus we see that all the conditions of  \cref{le:ode_approximation} are satisfied. By passing to the limit as $\eps\to0$ in \eqref{prelimpde} and using continuity of $g$, we get \eqref{eq:Iq_ode}.
\end{proof}

As we mentioned before, we are planning to apply \cref{L:mainlemma} recursively starting with $n=0$. 
To verify condition \eqref{formeraQ} we will use the following result.

\begin{lemma}\label{L:310}
For any $n\ge0$ we have
\begin{equation}\label{goodbound}
\int_0^4\int_0^\pi u^{-n-1}\sin^2\alpha \,\phi(\alpha,u)\,d\alpha du=
\frac14\int_0^4\int_0^\pi u^{-n} \phi(\alpha,u)\,d\alpha du.
\end{equation}
In case both sides of this identity are finite,
for any $\delta>0$ there exists a sequence $(\eps_k)_{k\in\Z_+}\searrow 0$ such that 
\begin{equation}\label{eq:ass_I-1}
\eps_k^{-n}\int_\delta^{\pi-\delta} \phi(\alpha,\eps_k)\,d\alpha \to 0\quad\text{as $k\to\infty$}.
\end{equation} 	
\end{lemma}
Note that \eqref{goodbound} can be rewritten as
$$
\int_0^\pi  I_{n+1}(\alpha)\sin^2\alpha\,d\alpha =
\frac14\int_0^\pi I_n(\alpha)d\alpha.
$$

\begin{proof}
Fix arbitrary $\delta>0$. Integrating \eqref{eq:fp_ualpha} in $\alpha$ from $\delta$ to $\pi/2$ yields for any $u\in(0,4]$
\begin{equation}\label{diffeq}
\partial_u \int_\delta^{\pi/2} (u-4\sin^2\alpha)\phi(\alpha,u)\, d\alpha= 
\frac{(\kappa-4)}{u}\phi(\delta,u)\sin^3\delta\cos\delta +\frac{\kappa}{2u}\sin^4\delta\,\,\partial_\alpha\phi(\alpha,u)\bigg\rvert_{\alpha=\delta}.
\end{equation}
By \cref{T:main}, $\phi(\alpha,4)=0$ for any $\alpha\in(0,\pi/2)$. Fix now arbitrary $u_0\in(0,4]$
and denote 
\begin{equation*}
J(\alpha) \defeq I_1^{u_0}(\alpha) = \int_{u_0}^4 u^{-1}\phi(\alpha,u)du,\quad \alpha\in(0,\pi).
\end{equation*}
Integrating \eqref{diffeq} in $u$ from $u_0$ to $4$ we get
\begin{equation}\label{prelim}
\Bigl|\int_\delta^{\pi/2} (u_0-4\sin^2\alpha)\phi(\alpha,u_0)\, d\alpha\Bigr|\le C J(\delta)\delta^3 +C\delta^4 |J'(\delta)|.
\end{equation}

Let us pass to the limit in \eqref{prelim} as $\delta\to0$. Note that $A:=\int_0^1 J(\alpha)\,d\alpha$ is obviously finite. Hence, we can apply \cref{Cor:310} with $f=J$, $h(t)=1/t$. Then, there exists a sequence  $(\delta_k)_{k\in\Z_+}$, such that
\begin{equation*}
\delta_k\downarrow0,\,\,\, \delta_k J(\delta_k)\le 1,\,\,\, 
\delta_k^2 |J'(\delta_k)|\le 1
\end{equation*}
for all $k\in\Z_+$. Applying now \eqref{prelim} with $\delta=\delta_k$ and passing to the limit as $k\to\infty$, we get
\begin{equation*}
\int_0^{\pi/2} (u_0-4\sin^2\alpha)\phi(\alpha,u_0)\, d\alpha=0,
\end{equation*}
which by symmetry of $\phi$ implies 
\begin{equation*}
\int_0^{\pi} (u_0-4\sin^2\alpha)\phi(\alpha,u_0)\, d\alpha=0
\end{equation*}
for any $u_0\in(0,4]$. Dividing now this identity by $u_0^{n+1}$ and integrating in $u_0$, we get \eqref{goodbound}.

To show \eqref{eq:ass_I-1}, fix $\delta>0$. Assuming the left-hand side of \eqref{goodbound} to be finite, we get 
\begin{equation*}
\int_0^4 \int_\delta^{\pi-\delta} u^{-n-1}\phi(\alpha,u)\,d\alpha \,du\le
\frac{1}{\sin^2\delta}\int_0^4\int_\delta^{\pi-\delta} u^{-n-1}\sin^2\alpha\,\phi(\alpha,u)\,d\alpha \,du<\infty.
\end{equation*} 
Therefore there must exist a sequence of $\eps_k\searrow 0$ satisfying \eqref{eq:ass_I-1} because otherwise the left-hand side of the above inequality would be infinite.
\end{proof}

\begin{remark}
\cref{L:310} can be deduced from a general PDE argument \cite[Theorem~2.3.2 and inequality (2.3.2)]{BKRS}. Indeed, note that PDE \eqref{FPK} can be written as 
$$
\mathcal L^*\phi=0,
$$
where $\mathcal L=\frac\kappa{2u}\sin^4\alpha\,\d^2_{\alpha\alpha}+\frac{4+\kappa}{u}\cos\alpha\sin^3\alpha\,\d_\alpha+(4\sin^2\alpha-u)\,\d_u$, $u>0$, $\alpha\in(0,\pi)$. If $n=1$, take a Lyapunov function $V(\alpha,u):=-\log u$; otherwise set $V(\alpha,u):=\frac{1}{n-1}u^{-n+1}$. 
Then 
\begin{equation*}
\mathcal{L}V(\alpha,u)=(u-4\sin^2\alpha)u^{-n}.
\end{equation*}
Note however that even though $V$ does not satisfy all the conditions of \cite[Theorem~2.3.2]{BKRS}, a standard mollification argument and \cite[inequality (2.3.2)]{BKRS} yield
\eqref{goodbound}. However, writing up rigorously all the technical details gets a bit complicated, so we found it simpler to give a direct proof.
\end{remark}

We are now able to prove \eqref{mainidentity} rigorously. 

\begin{proof}[Proof of \cref{L:db1}]
Let us apply \cref{L:mainlemma} with $n=0$. We see that the right-hand side of \eqref{goodbound} is finite for $n=0$. Hence \cref{L:310} implies that\eqref{eq:ass_I-1} holds for $n=0$. Therefore, condition \eqref{formeraQ} is satisfied for $n=0$.

Note now that if $I_{1}=\infty$ for all $\alpha\in(0,\pi)$, then the left-hand side of \eqref{goodbound} with $n=0$ is infinite. However this is not the case. Thus, by  \cref{L:mainlemma}, the function $I_1$ is twice differentiable and solves
\begin{equation*}
(\kappa-4)(3\sin^2\alpha\cos^2\alpha-\sin^4\alpha) \,I_1+
(3\kappa-4)\sin^3\alpha\cos\alpha \,I_1'	+\frac{\kappa}{2}\sin^4\alpha \,I_1''=0.
\end{equation*}
This can be rewritten as 
\begin{equation}\label{denssetp1}
\partial_\alpha\Bigl((\kappa-4)\sin^3\alpha\cos\alpha\, I_1+\frac{\kappa}{2}\sin^4\alpha\, I_1'\Bigr)=0.
\end{equation} 

Let $\alpha \in (0,\pi)$. Then integrating \eqref{denssetp1} in $\alpha' \in [\alpha,\pi-\alpha]$, we get
\begin{equation}\label{denssetp2}
(\kappa-4)\sin^3\alpha\cos\alpha (I_1(\alpha)+I_1(\pi-\alpha))+\frac{\kappa}{2}\sin^4\alpha (I_1'(\alpha)-I_1'(\pi-\alpha))=0
\end{equation}
Recall that by \cref{T:main} we have that the density $\psi$ is symmetric, $\psi(x,y)=\psi(-x,y)$ for $x\in\R$, $y>0$. This implies that 
$\phi$ is also symmetric and $\phi(\alpha,u)=\phi(\pi-\alpha,u)$, 
$\d_\alpha\phi(\alpha,u)=-\d_\alpha\phi(\pi-\alpha,u)$ for $\alpha\in(0,\pi)$, $u>0$. Hence
$I_1(\alpha)=I_1(\pi-\alpha)$,  $I_1'(\alpha)=-I_1'(\pi-\alpha)$ and 
\eqref{denssetp2} yields 
\begin{equation}\label{denssetp3}
(\kappa-4)\sin^3\alpha\cos\alpha\, I_1(\alpha)+\frac{\kappa}{2}\sin^4\alpha\, I_1'(\alpha)=0.
\end{equation}
Therefore,
$$
\d_\alpha(\sin^{2-8/\kappa}\alpha\, I_1(\alpha))=0,
$$
and we finally get  
\begin{equation*}
I_1(\alpha)=c\sin^{8/\kappa-2}\alpha,\quad \alpha\in(0,\pi),
\end{equation*}
for some $c>0$. The precise value of $c$ follows from \eqref{goodbound}:
\begin{equation*}
\frac14=\int_0^4\int_0^\pi \frac1u\sin^2\alpha\, \phi(\alpha,u)\,d\alpha\, du=c\int_0^\pi \sin^{8/\kappa}\alpha \,d\alpha = c \sqrt{\pi}\frac{\Gamma(\frac12+\frac4\kappa)}{\Gamma(1+\frac4\kappa)},
\end{equation*}
which gives \eqref{mainidentity}. 
\end{proof}

\begin{proof}[Proof of \cref{T:other}(i)]
Note that the $\Law (\gamma(t))=\Law(t^{1/2}\gamma(1))$. Therefore
\begin{equation}\label{peng1}
\E\int_0^\infty \ind_{\gamma(t) \in \Lambda} \, dt
	= \int_0^\infty \pr(\gamma(t) \in \Lambda) \,dt 
	= \int_0^\infty \pr(\gamma(1) \in t^{-1/2}\Lambda) \,dt .
\end{equation}
Fix $0<a<b$, $0<\alpha<\beta<\pi$. First consider sets $\Lambda$ of the form 
\begin{equation}\label{formu}
\Lambda = \{x+iy \mid \cot^{-1}(x/y) \in [\alpha,\beta],\ y^2 \in [a,b] \}.
\end{equation}
Then, writing $\gamma(1) = \sqrt U(\cot A+i)$, we continue \eqref{peng1} in the following way
\begin{align*}
\E\int_0^\infty \ind_{\gamma(t) \in \Lambda} \, dt&=
\int_0^\infty \P(A\in[\alpha,\beta], U \in [a/t,b/t]) \,dt\\
&= 
\int_0^\infty\int_\alpha^\beta\int_{a/t}^{b/t} \phi(\alpha',u)\,d u d\alpha'  dt\\
&=\int_0^\infty\int_\alpha^\beta \frac{b-a}{u} \phi(\alpha',u)\,d\alpha' du\\
&=(b-a)\frac{\Gamma(1+\frac4\kappa)}{4\sqrt\pi\Gamma(\frac12+\frac4\kappa)}\int_\alpha^\beta (\sin\alpha')^{8/\kappa-2} \,d\alpha',
\end{align*}
where the last identity follows from \cref{L:db1}. Since
\begin{align*}
\int_\Lambda \bigl( 1+\frac{x^2}{y^2} \bigr)^{-4/\kappa} \,dx\,dy
&= \frac12\int_\alpha^\beta\int_a^b  (\sin \alpha')^{8/\kappa-2}\,du d\alpha'\\
&= \frac{b-a}{2} \int_\alpha^\beta  (\sin \alpha')^{8/\kappa-2}\,d\alpha',	
\end{align*}
we see that
\begin{equation*}
\E\int_0^\infty \ind_{\gamma(t) \in \Lambda} \, dt=\frac{\Gamma(1+\frac4\kappa)}{2\sqrt\pi\Gamma(\frac12+\frac4\kappa)}	\int_\Lambda \bigl( 1+\frac{x^2}{y^2} \bigr)^{-4/\kappa} \,dx\,dy.
\end{equation*}
Clearly, sets $\Lambda$ of the form \eqref{formu} generate the Borel $\sigma$-algebra on $\HH$. This implies \eqref{eq:zhan_law}.
\end{proof}

To prove \cref{T:other}(ii), we need the following key result.
\begin{lemma}\label{le:I-2_asymptotics}
Let $n\in\Z_+$, $n\ge1$. Then $\displaystyle\int_0^\pi I_n(\alpha)\,d\alpha$ is finite for $\kappa<8/(2n-1)$ and infinite for $\kappa\ge8/(2n-1)$.

Furthermore, let $\frac8\kappa>2n-3$. Then the function $I_n\colon(0,\pi)\to\R_+$ is continuous and for any $\delta>0$ there exists $\alpha_0=\alpha_0(n,\delta)\in(0,\pi/2)$ such that for $\alpha\in(0,\alpha_0)$
\begin{equation}\label{eqlim}
I_n(\alpha)\ge \alpha^{8/\kappa-2n}\abs{\log\alpha}^{-\delta}.
\end{equation}
If, additionally, $\frac8\kappa>(2n-3)$ and $\kappa<\frac{16}3$, then for any $\delta>0$ there exists $\alpha_0=\alpha_0(n,\delta)\in(0,\pi/2)$ such that for $\alpha\in(0,\alpha_0)$ 
\begin{equation}\label{eqlim2}
I_n(\alpha)\le \alpha^{8/\kappa-2n-\delta}.
\end{equation}
\end{lemma}

\begin{proof}
We will prove this lemma by induction over $n$, with the case $n=1$ already established in \cref{L:db1}. 
Let us first explain the heuristic idea. Consider for simplicity the first non-trivial case $n=2$. Then approximating \eqref{eq:Iq_ode} near $\alpha \approx 0$ and knowing that $I_{1} = c_0(\sin\alpha)^{8/\kappa-2}$, the equation reads
	\[
	0 \approx c_0\alpha^{8/\kappa-2} - 4\alpha^2 \,I_{2}
	+(\kappa-4)3\alpha^2 \,I_{2} 
	+(3\kappa-4)\alpha^3 \,I_{2}'
	+\frac{\kappa}{2}\alpha^4 \,I_{2}'' .
	\]
	If we naively suppose $I_{2} \approx \alpha^s$, $I_{2}' \approx s\alpha^{s-1}$, $I_{2}'' \approx s(s-1)\alpha^{s-2}$, then we find that either $s = 8/\kappa-4$, cancelling the first term $I_{1}$, or $s < 8/\kappa-4$ in which case the remaining terms need to cancel each other. In the latter case, the coefficients need to sum to $0$, i.e.
\begin{equation}\label{eqfors}
 0 = (3\kappa-16)+(3\kappa-4)s+\frac{\kappa}{2}s(s-1).
\end{equation}
	Recall also, that by \cref{L:310} with $n=1$, we have $\int_0^\pi I_2 \sin^2\alpha\,d\alpha<\infty$, which implies $s>-3$. However, on the interval $(-3, 8/\kappa-4)$ equation \eqref{eqfors} has no solutions, and thus the case $s < 8/\kappa-4$ is not possible. Hence, the only remaining option is  $s = 8/\kappa-4$.
	
	To make this heuristic precise, we find a suitable subsequence $\alpha_k \searrow 0$ where we can apply a similar argument. 
	
Let us now proceed to the rigorous induction on $n$.

\noindent \textbf{Base case}. $n=1$. In this case \eqref{eqlim}, \eqref{eqlim2} and continuity of $I_1$ was already proven in \eqref{mainidentity}. The fact that $\int_0^\pi I_1(\alpha)\,d\alpha$ is finite if and only if $\kappa<8$ is immediate.

\noindent \textbf{Inductive step}. Suppose that the statement of the lemma is valid for $n\in\Z_+$. Let us prove it for $n+1$. 

If $\kappa\ge 8/(2n-1)$, then $\int_0^\pi I_n(\alpha)\,d\alpha=\infty$, and this obviously implies that  $\int_0^\pi I_{n+1}(\alpha)\,d\alpha=\infty$. Therefore it is sufficient to consider the case $\kappa<8/(2n-1)$. By the inductive step, for these values of $\kappa$ we have $\int_0^\pi I_n(\alpha)\,d\alpha<\infty$. Hence, \cref{L:310} implies that condition \eqref{eq:ass_I-1} holds. This, together with continuity of $I_n$, shows that all the conditions of \cref{L:mainlemma} are met. Note that we cannot have $I_{n+1}=\infty$ for all $\alpha\in(0,\pi)$. Indeed, in this case the left-hand side of identity \eqref{goodbound} would be infinite but the right-hand side of this identity is finite (because it is equal to $C\int_0^\pi I_n(\alpha)\,d\alpha$).  
Thus, \cref{L:mainlemma} implies that 
\begin{equation}\label{td}
\text{$I_{n+1}$ is twice differentiable and satisfies 
\eqref{eq:Iq_ode}}.
\end{equation}
Using this, we now show \eqref{eqlim} and \eqref{eqlim2}. The statement about the finiteness of $\int I_n\,d\alpha$ follows immediately.

\textbf{Lower bound.}
We begin with the lower bound   \eqref{eqlim}. Denote
\begin{equation}\label{defs}
s:=\frac{8}{\kappa}-2n-2 .
\end{equation}
Fix $\delta\in(0,1)$ and suppose that the lower bound does not hold, i.e. we have $I_{n+1}(\wt\alpha_k) < \tilde\alpha_k^s\abs{\log\tilde\alpha_k}^{-\delta}$ for a sequence of $\wt\alpha_k \searrow 0$. We distinguish two cases.

\textbf{Case 1.1.} $I_{n+1}(\alpha) \le \alpha^s\abs{\log\alpha}^{-\delta}$ for all small $\alpha>0$. We apply \cref{le:nice_subseq} with $f(\alpha): = \partial_\alpha(\alpha^{-s}I_{n+1}(\alpha))$, $h(\alpha) = \alpha^{-1}\abs{\log\alpha}^{-\delta}$. It is easy to see that all the conditions of the lemma are satisfied, 
and therefore there exists a sequence of $\alpha_k \searrow 0$ such that (for some $C<\infty$)
\begin{align*}
\abs{I_{n+1}(\alpha_k)} &\le \alpha_k^s\abs{\log\alpha_k}^{-\delta},\\
\abs{I_{n+1}'(\alpha_k)} &\le C\alpha_k^{s-1}\abs{\log\alpha_k}^{-\delta},\\
I_{n+1}''(\alpha_k) &\ge -C\alpha_k^{s-2}\abs{\log\alpha_k}^{-\delta}.
\end{align*}
Plugging this into \eqref{eq:Iq_ode} we derive
\begin{align}
0 &=nI_n(\alpha_k) +\sin^2\alpha_k(3\kappa-4n-12+(16-4\kappa)\sin^2\alpha_k) \,I_{n+1}(\alpha_k)\nn \\
&\hphantom{=}+(3\kappa-4)\sin^3\alpha_k\cos\alpha_k\,I_{n+1}'(\alpha_k)
+\frac{\kappa}{2}\sin^4\alpha_k \,I_{n+1}''(\alpha_k) \nn\\
&\ge nI_n (\alpha_k)+\sin^2\alpha_k[(3\kappa-4n-12+(16-4\kappa)\sin^2\alpha_k)\wedge0]
\,\alpha_k^s\abs{\log\alpha_k}^{-\delta} \nn\\
&\hphantom{=}-C \sin^3\alpha_k\cos\alpha_k\, \alpha_k^{s-1}\abs{\log\alpha_k}^{-\delta}
-C \sin^4 \alpha_k\, \alpha_k^{s-2}\abs{\log\alpha_k}^{-\delta}\label{lowerbound}.
\end{align}
Multiplying \eqref{lowerbound} by $\alpha^{-s-2}\abs{\log\alpha}^\delta$ and passing to the limit as  $\alpha_k \searrow 0$, we get
\begin{equation}\label{lowerbound2}
-C+n\liminf_{\alpha\to0} \frac{I_n(\alpha)\abs{\log\alpha}^\delta}{\alpha^{s+2}}\le0
\end{equation}
By induction hypothesis (applied with $\delta/2$ in place of $\delta$),  $I_n(\alpha)\alpha^{-s-2}\abs{\log\alpha}^\delta \to \infty$. This contradicts \eqref{lowerbound2}.
Therefore it cannot be that $I_{n+1}(\alpha) \le \alpha^s\abs{\log\alpha}^{-\delta}$ for all small $\alpha$.

\textbf{Case 1.2.} In the other case one can find two sequences  $\wt\alpha_k, \dbtilde\alpha_k \searrow 0$ such that $\dbtilde\alpha_{k+1}\le \wt\alpha_k \le \dbtilde\alpha_{k}$ and 
$I_{n+1}(\wt\alpha_k) < \tilde\alpha_k^s\abs{\log\tilde\alpha_k}^{-\delta}$,  $I_{n+1}(\dbtilde\alpha_k) > \dbtilde\alpha_k^s\abs{\log\dbtilde\alpha_k}^{-\delta}$. Pick $\alpha_k \in \argmin_{[\dbtilde\alpha_{k+1},\dbtilde\alpha_{k}]}(I_{n+1}(\alpha)-\alpha^s\abs{\log\alpha}^{-\delta}$). Then $\alpha_k\in(\dbtilde\alpha_{k+1},\dbtilde\alpha_{k})$ and 
\begin{align*}
\abs{I_{n+1}(\alpha_k)} &< \alpha_k^s\abs{\log\alpha_k}^{-\delta},\\
I_{n+1}'(\alpha_k) &= s\alpha_k^{s-1}\abs{\log\alpha_k}^{-\delta}+o(\alpha_k^{s-1}\abs{\log\alpha_k}^{-\delta}),\\
I_{n+1}''(\alpha_k) &\ge s(s-1)\alpha_k^{s-2}\abs{\log\alpha_k}^{-\delta}+o(\alpha_k^{s-2}\abs{\log\alpha_k}^{-\delta}).
\end{align*}
This implies that \eqref{lowerbound} holds for this sequence $(\alpha_k)$, which again leads to a contradiction. 

Thus, we have shown that $I_{n+1}(\alpha)\ge\alpha^{s}\abs{\log\alpha}^{-\delta}$ for all small enough $\alpha$. Recalling the definition of $s$ in \eqref{defs}, we see that this is exactly the desired lower bound in \eqref{eqlim}. This bound implies that for $\kappa\ge 8/(2(n+1)-1)=8/(2n+1)$ we have $\int_0^\pi I_{n+1}(\alpha)\,d\alpha=\infty$.

\textbf{Upper bound.}
Now we proceed with the upper bound in \eqref{eqlim2}. We suppose now that $\kappa<\frac8{2n-1}\wedge\frac{16}3$. We use again notation \eqref{defs}. Write $I_{n+1}(\alpha) =: \alpha^{s(\alpha)}$ for $\alpha \in (0,\pi)$. We will distinguish two cases.

\textbf{Case 2.1.} Suppose that 
\begin{equation*}
\liminf_{\alpha\to0} s(\alpha)<\limsup_{\alpha\to0} s(\alpha).
\end{equation*}
We show that this is impossible by deriving a contradiction.

Note that by Step~1, $\limsup_{\alpha\to0} s(\alpha)\le s$. Further, there exists  a sequence $\beta_k\searrow0$, such that $s(\beta_k)>-3$. Indeed, otherwise the left-hand side of identity \eqref{goodbound} would be infinite whilst the right-hand side of this identity is finite thanks to the induction hypothesis.
Therefore, by continuity of $s(\alpha)$ there exists $r\in[-3,s)$, and sequences  $\wt\alpha_k, \dbtilde\alpha_k \searrow 0$ such that $\dbtilde\alpha_{k+1}\le \wt\alpha_k \le \dbtilde\alpha_{k}$ and 
$s(\wt \alpha_{k})< r$,  $s(\dbtilde \alpha_{k})> r$. Pick now 
\[
\alpha_k \in \argmax_{[\dbtilde\alpha_{k+1},\dbtilde\alpha_{k}]}(I_{n+1}(\alpha)-\alpha^{r}).
\]
Then $\alpha_k\in(\dbtilde\alpha_{k+1},\dbtilde\alpha_{k})$ and 
\begin{align*}
I_{n+1}(\alpha_k) &> \alpha_k^{r},\\
I_{n+1}'(\alpha_k) &= r\alpha_k^{r-1},\\
I_{n+1}''(\alpha_k) &\le r(r-1)\alpha_k^{r-2}.
\end{align*}
Substituting this into \eqref{eq:Iq_ode}, dividing it by $\alpha_k^{r+2}$ and letting $\alpha_k \searrow 0$, we get
\begin{equation*}
(3\kappa-4n-12)+(3\kappa-4) r +\frac\kappa2 r(r-1)+
n\limsup_{\alpha\to0} \frac{I_n(\alpha)}{\alpha^{r+2}}\ge0.
\end{equation*}
(Here we have used $\kappa \le 16/3$, implying $3\kappa-4n-12+(16-4\kappa)\sin^2\alpha_k < 0$).

Since $r+2 < s+2 = 8/\kappa-2n$, we have $ \frac{I_n(\alpha)}{\alpha^{r+2}}\to0$ by the induction hypothesis. Hence,
\begin{equation}\label{feqc21}
(3\kappa-4n-12)+(3\kappa-4) r +\frac\kappa2 r(r-1)\ge0.
\end{equation}
Recall that $r\in[-3,s)$. Note that the left-hand side of the above expression is strictly negative for $r=-3$ and for $r=s$; in the latter case it equals $n (\kappa (2 n - 1) - 12) < 0$ thanks to our standing assumption $\kappa<8/(2n-1)$. Hence the left-hand side of \eqref{feqc21} is strictly negative for any  $r\in[-3,s)$ which is a contradiction. 

\textbf{Case 2.2.} It follows from above that $\liminf_{\alpha\to0} s(\alpha)=\limsup_{\alpha\to0}s(\alpha)=:r$ and $r\in[-3,s]$. We would like to show $r=s$ which is \eqref{eqlim2}.

Suppose $r<s$. Note that \eqref{td} implies that $s(\alpha)$ is twice differentiable.
Therefore, all the conditions of \cref{le:nice_subseq} are satisfied for the functions  $f(\alpha):=s'(\alpha)$, $h(\alpha):=\frac{1}{\alpha\abs{\log\alpha}\log\abs{\log\alpha}}$. 
Thus there exists a sequence $\alpha_k\searrow0$ such that $|s'(\alpha_k)|\le \frac{1}{\alpha_k\abs{\log\alpha_k}\log\abs{\log\alpha_k}}$ and 
$s''(\alpha_k)\ge -\frac{1}{\alpha_k^2\abs{\log\alpha_k}\log\abs{\log\alpha_k}}$.

Recalling that 
	\begin{align*}
		I_{n+1}'(\alpha) &= \left(\frac{s(\alpha)}{\alpha}+s'(\alpha)\log\alpha\right)\alpha^{s(\alpha)},\\
		I_{n+1}''(\alpha) &= \left(-\frac{s(\alpha)}{\alpha^2}+2\frac{s'(\alpha)}{\alpha}+s''(\alpha)\log\alpha+\left(\frac{s(\alpha)}{\alpha}+s'(\alpha)\log\alpha\right)^2\right)\alpha^{s(\alpha)},
	\end{align*}
we get
\begin{align*}
I_{n+1}(\alpha_k) &= \alpha_k^{s(\alpha_k)},\\
I_{n+1}'(\alpha_k)&=(s(\alpha_k)+o(1))\alpha_k^{s(\alpha_k)-1},\\
I_{n+1}''(\alpha_k) &\le (-s(\alpha_k)+s(\alpha_k)^2+o(1))\alpha_k^{s(\alpha_k)-2},
\end{align*}
where $o(1)$ denote some sequences that tend to $0$ as $k \to \infty$.
Now we substitute this into \eqref{eq:Iq_ode}, divide it by $\alpha_k^{s(\alpha_k)+2}$ and let $\alpha_k \searrow 0$.
We derive
\begin{equation}\label{result4}
(3\kappa-4n-12)+(3\kappa-4)r+\frac{\kappa}{2}r(r-1)+n\limsup_{\alpha\to0} \frac{I_n(\alpha)}{\alpha^{s(\alpha)+2}}\ge0.
\end{equation} 
If now $r<s$, then there exists $\delta > 0$ such that $s(\alpha)\le r+\delta<s$ for all $\alpha$ small enough. Hence, thanks to the induction hypothesis, $\limsup_{\alpha\to0} \frac{I_n(\alpha)}{\alpha^{s(\alpha)+2}}=0$. 
Therefore inequality \eqref{feqc21} holds for a certain $r\in[-3,s)$ which is a contradiction as before. 

Thus we have shown that $\liminf_{\alpha\to0}s(\alpha)=s$. Therefore, $I_{n+1}(\alpha)=\alpha^{s(\alpha)}\le \alpha^{s-\delta}$ for all $\alpha$ small enough, so the upper bound \eqref{eqlim2} holds. Hence for $\kappa< 8/(2(n+1)-1)=8/(2n+1)$ we have $\int_0^\pi I_{n+1}(\alpha)\,d\alpha<\infty$.
\end{proof}

Now we are ready to complete the proof of \cref{T:other}

\begin{proof}[Proof of  \cref{T:other}(ii)]
Inequality \eqref{ineq24} follows directly from \cref{le:I-2_asymptotics} and the definition of $I_n$. Further, for $\kappa<8$ we have from \cref{L:db1}:
\begin{equation*}
\E (\Im \gamma_1)^{-2}=\int_0^\pi\int_0^4 \frac1u\phi(\alpha,u)\,d\alpha du =
\frac{2}{8-\kappa},
\end{equation*}
which is \eqref{ineq25}.

To show \eqref{ineq26}, fix $\kappa<8/3$. Note that in this regime by \cref{le:I-2_asymptotics}, we have $\int_0^1 I_2(\alpha)<\infty$, and thus by \cref{Cor:310} with $f=I_2$, $h=1/(\alpha |\log \alpha|)$ there exists a sequence $\alpha_k\searrow 0$ such that 
\begin{align}
\lim_{\alpha_k\searrow 0} \alpha_k I_{2}(\alpha_k) = 0 ,\label{lim41}\\
\lim_{\alpha_k\searrow 0} \alpha_k^2 I_{2}'(\alpha_k) = 0\label{lim42} .
\end{align}
It was shown in the proof of \cref{le:I-2_asymptotics}, that in this case $I_2$ satisfies  \eqref{eq:Iq_ode} which can be rewritten as
\begin{equation}\label{eq:Iq_ode2}
\frac{I_1}{\sin^2\alpha} - 4 I_{2}
		+(\kappa-4)(3-4\sin^2\alpha) I_{2} 
		+(3\kappa-4)\sin\alpha\cos\alpha\, I_{2}'
		+\frac{\kappa}{2}\sin^2\alpha\, I_{2}''=0.
\end{equation}
Integrate now the above equation in $\alpha$ from $\alpha_k$ to $\pi-\alpha_k$, then integrate by parts. Thanks to \eqref{lim41} and \eqref{lim42}, all the boundary terms vanish when we send $\alpha_k \searrow 0$.
Note also that by \cref{L:310}, we have $\int_0^\pi I_2(\alpha)\sin^2\alpha\,d\alpha=\frac14 \int_0^\pi I_1(\alpha)\,d\alpha$. We get
\begin{equation}\label{relation}
\int_0^\pi \frac{I_{1}(\alpha)}{\sin^2 \alpha}\,d\alpha+(\kappa-12)\int_0^\pi I_{2}(\alpha)\,d\alpha +2\int_0^\pi I_{1}(\alpha)\,d\alpha=0.
\end{equation}
Recalling the expression for $I_1$ from \cref{L:db1}, we deduce
\begin{equation*}
\E(\Im\gamma_1)^{-4} = \int_0^\pi I_2(\alpha)\,d\alpha  = \frac{48-16\kappa}{(12-\kappa)(8-\kappa)(8-3\kappa)}.\qedhere
\end{equation*}
\end{proof}
\begin{remark}
For general $n$, the identity \eqref{relation} reads
\[ (4n+8-\kappa) \int_0^\pi I_{n+1}(\alpha)\,d\alpha = n \int_0^\pi \frac{I_n(\alpha)}{\sin^2 \alpha}\,d\alpha + 2 \int_0^\pi I_n(\alpha)\,d\alpha . \]
Unfortunately, we do not have an explicit formula for $\int_0^\pi \frac{I_{n}(\alpha)}{\sin^2 \alpha}\,d\alpha$ for $n\ge2$. This prevents us from getting explicit formulas of negative moments of $\Im(\gamma_1)$ of higher order.
\end{remark}

\begin{remark}
Another possible approach to find explicit formulas for $I_n$ would be through its Fourier coefficients
\[ a_0 = \frac{1}{\pi}\int_0^\pi I_n\,d\alpha , \quad a_j = \frac{2}{\pi}\int_0^\pi I_n\cos(2j\alpha)\,d\alpha . \]
Formally expanding \eqref{eq:Iq_ode2}, we obtain a (countable) system of linear equations for $(a_j)_{j \ge 0}$ in terms of the Fourier coefficients $(b_j)_{j \ge 0}$ of the function $I_{n-1}/\sin^2\alpha$. However, it seems difficult to solve the system of equations explicitly. Only for $a_0,a_1$ we get a system of two equations in terms of $b_0,b_1$ which correspond exactly to what we obtain from the proof above.
\end{remark}


\end{document}